\documentclass[reqno,final]{amsart}
\RequirePackage[numbers, sort&compress]{natbib} 
\usepackage{fancyhdr} 
\usepackage{color} 
\usepackage{hyperref} 
\usepackage{graphicx} 

\usepackage{pstricks}
\usepackage{amssymb}
\usepackage{verbatim}

\definecolor{aleacolor}{rgb}{0.16,0.59,0.78}

\hypersetup{
breaklinks,
colorlinks=true,
linkcolor=aleacolor,
urlcolor=aleacolor,
citecolor=aleacolor}

\renewcommand{\headheight}{11pt}
\renewcommand{\cite}{\citet}

\theoremstyle{plain}
\newtheorem{theorem}{Theorem}[section]                                          
\newtheorem{proposition}[theorem]{Proposition}                          
\newtheorem{lemma}[theorem]{Lemma}
\newtheorem{corollary}[theorem]{Corollary}
\newtheorem{algothm}[theorem]{Algorithm/Theorem}
\newtheorem{conjecture}[theorem]{Conjecture}
\theoremstyle{definition}
\newtheorem{definition}[theorem]{Definition}
\theoremstyle{remark}
\newtheorem{remark}[theorem]{Remark}
\newtheorem{example}[theorem]{Example}

\numberwithin{equation}{section}

\newtheorem{algorithm}[theorem]{Algorithm}

\makeatletter \@addtoreset{equation}{section} \makeatother
\renewcommand\theequation{\thesection.\arabic{equation}}
\renewcommand\thefigure{\thesection.\arabic{figure}}
\renewcommand\thetable{\thesection.\arabic{table}}

\RequirePackage[OT1]{fontenc}
\RequirePackage{amsthm,amsmath, amssymb, latexsym}

\begin{document}

\title{Branch merging on continuum trees with applications to regenerative tree growth}

\author{Franz Rembart}

\address{Department of Statistics \newline University of Oxford\newline
24-29 St Giles' \newline
Oxford OX1 3LB \newline
United Kingdom}

\email{franz.rembart@stats.ox.ac.uk}

\thanks{Research supported by EPSRC grant EP/P505666/1.}


\begin{abstract}
 We introduce a family of branch merging operations on continuum trees and show that Ford CRTs are distributionally invariant. This operation is new even in the special case of the Brownian CRT, which we explore in more detail. The operations are based on spinal decompositions and a regenerativity preserving merging procedure of $(\alpha,\theta)$-strings of beads, that is, random intervals $[0,L_{\alpha,\theta}]$ equipped with a random discrete measure $dL^{-1}$ arising in the limit of ordered $(\alpha,\theta)$-Chinese restaurant processes as introduced by Pitman and Winkel. Indeed, we iterate the branch merging operation recursively and give a new approach to the leaf embedding problem on Ford CRTs related to $(\alpha,2-\alpha)$-tree growth processes.
\end{abstract}

\maketitle

\section{Introduction}
\label{Intro}
Tree-valued Markov processes and operations on continuum random trees (CRTs) such as pruning have recently attracted particular interest in probability theory. Evans and Winter, for instance, consider regrafting in combination with subtree pruning and show in \citep{10} that Aldous' Brownian CRT arises as the stationary distribution of a certain reversible $\mathbb R$-tree valued Markov process. \cite{6} present a similar study involving root growth with regrafting. \cite{58,59} and \cite{35,57} consider a Markov chain operating on the space of binary rooted $\mathbb R$-trees randomly removing and reinserting leaves, related to a diffusion limit on the space of continuum trees. See also \citep{50,51,52,53,55,63,64,65} for related work.

We study branch merging as a new operation on (continuum) trees, which leaves Ford CRTs, and in particular the Brownian CRT, distributionally invariant.  
Our branch merging operation is based on a leaf sampling procedure, and on the study of the resulting reduced subtrees equipped with projected subtree masses. The notion of a string of beads naturally captures this projected mass on a branch.

Following \cite{1}, we consider $(\alpha, \theta)$-strings of beads for $\alpha \in (0,1)$ and $\theta > 0$. An  $(\alpha, \theta)$-string of beads is a random interval $[0, L_{\alpha, \theta}]$ equipped with a random discrete measure $dL^{-1}$, arising in the framework of the two-parameter family of $(\alpha, \theta)$-Chinese restaurant processes (CRP) when equipped with a regenerative table order. An $(\alpha, \theta)$-Chinese restaurant process as introduced by Dubins and Pitman (see e.g. \citep{5}) is a sequence of exchangeable random partitions $(\Pi_n, n \geq 1)$ of $[n]:=\{1,\ldots,n\}$ defined by customers labelled by $[n]$ sitting at a random number of tables. Customer $1$ sits at the first table, and, at step $n+1$, conditionally given $k$ tables in the restaurant with $n_1, \ldots, n_k$ customers at each table, customer $n+1$ 
\begin{itemize}
\item sits at the $i$-th occupied table with probability $(n_i-\alpha)/(n+\theta)$, $i \in [k]$;
\item opens a new table with probability $(k\alpha+\theta)/(n+\theta)$. 
\end{itemize}

The joint distribution of the asymptotic relative table sizes arranged in decreasing order is known to be Poisson-Dirichlet with parameters $(\alpha, \theta)$, for short PD$(\alpha,\theta)$. PD$(\alpha,\theta)$ vectors have been widely studied in the literature: see e.g. \citep{17, 39} for constructions and properties, \citep{16} and Section 10.3 in \citep{5} for coagulation-fragmentation-dualities, and \citep{1} for Poisson-Dirichlet compositions arising in the framework of continuum random trees.

\cite{1} equipped the $(\alpha, \theta)$-CRP with a table order which is independent of the actual seating process. These structures arise naturally in the study of the $(\alpha, \theta)$-tree growth process $(T_n^{\alpha, \theta}, n \geq 1)$, that is, a rooted binary regenerative tree growth process with $n$ leaves at step $n$ labelled by $[n]$ such that the partition structure induced by the subtrees (tables) encountered on the unique path from the root to the first leaf defines an ordered $(\alpha, \theta)$-CRP, see Section \ref{alphathetamodel}. 

It was shown that the table sizes in the ordered $(\alpha, \theta)$-CRP induce a regenerative composition structure of $[n]$ in the sense of \cite{3}, and yield a so-called $(\alpha, \theta)$-regenerative interval partition of $[0,1]$ in the scaling limit. The lengths of the components of this interval partition are the limiting proportions of table sizes in the CRP arranged in regenerative order, and correspond to the masses of atoms of the discrete measure $dL^{-1}$ associated with an $(\alpha, \theta)$-string of beads. $(\alpha, \theta)$-strings of beads in this sense align the limiting table proportions of an $(\alpha, \theta)$-CRP along an interval of length $L_{\alpha, \theta}$ in regenerative order where  $L_{\alpha, \theta}$ is the almost-sure limit as $n\rightarrow \infty$ of the number of tables in the CRP at step $n$ scaled by $n^{-\alpha}$ (see Section \ref{crp}).

In this article, we study a merging operation for $(\alpha, \theta_i)$-strings of beads, $i \in [k]$ for some $k \in \mathbb N$ fixed, subject to a Dirichlet$(\theta_1, \ldots, \theta_k)$ mass split with $\alpha \in (0,1)$, $\theta_1, \ldots, \theta_k >0$. This situation arises naturally in the framework of ordered CRPs. Consider $k$ ordered CRPs with parameters $(\alpha, \theta_i)$, $i \in [k]$. At each step, let one of these restaurants be selected proportionally to the total weight in each CRP with initial weights $\theta_i, i \in [k]$. Conditionally given that the $j$-th restaurant is selected with $m_j(n)$ customers present, the $(n+1)$st costumer is customer $m_j(n)+1$ in the $(\alpha, \theta_j)$-CRP, and chooses a table according to the $(\alpha, \theta_j)$-selection rule, independent of the restaurant selection process. Then, by a simple urn model, we obtain a Dirichlet$(\theta_1, \ldots, \theta_k)$ mass split between these restaurants in the limit, and each CRP is associated with an $(\alpha, \theta_i)$-string of beads, when we rescale distances and masses by appropriate scaling factors depending on the Dirichlet vector. Remarkably, the identified $(\alpha, \theta_i)$-strings of beads are independent of each other and of the Dirichlet$(\theta_1, \ldots, \theta_k)$ mass split. 
We ask the following questions.
\begin{itemize}
\item (How) can we merge these ordered $(\alpha, \theta_i)$-CRPs, $i \in [k]$, by completing the partial table orders to obtain the seating rule of an ordered $(\alpha, \theta)$-Chinese restaurant process for some $\theta >0$?
\item How does the parameter $\theta$ depend on $\alpha$ and $\theta_i$, $i \in [k]$?
\end{itemize}

Our merging algorithm allows to merge $(\alpha, \theta_i)$-strings of beads preserving the regenerative property. It becomes particularly simple when $\alpha=\theta_i=1/2$ for all $i \in [k]$, which is the case relevant for branch merging on the Brownian CRT (see Figure \ref{figure1} for an example with $k=3$).
\begin{algothm}[Merging $(1/2, 1/2)$-strings of beads]  \label{algthm}  Let $E_i=(\rho_i, \rho_i')$, $i \in [k]$, be $k$ disjoint intervals and $\mu$ a mass measure on $E:=\bigcup_{i=1}^k E_i$ such that $(\mu(E_1), \ldots, \mu(E_k)) \sim  \text{\rm Dirichlet}(1/2, \ldots, 1/2)$. Suppose that the rescaled pairs 
$(\mu(E_i)^{-1/2}E_i, \mu(E_i)^{-1} \mu \restriction_{E_i})$ are $(1/2, 1/2)$-strings of beads, $i \in [k]$, independent of each other and of the Dirichlet mass split $(\mu(E_1), \ldots, \mu(E_k))$. 

Define the metric space $(E_{\biguplus}, d_{\biguplus})$ in the following way.
\begin{itemize}
\item Set $E^{(1)}:=E$, $E^{(1)}_i:=E_i$, $\rho_i^{(1)}:=\rho_i$, $i \in [k]$. For $n \geq 1$, do the following.
\begin{enumerate}
\item[(i)] Select an atom $X_n$ from $\mu \restriction_ {E^{(n)}}$ with probability proportional to mass.
\item[(ii)] Let $a_n:= \rho^{(n)}_{I_n}$ and $b_{n}:=X_n$ where $I_n:=i$ if $X_n \in E_i$, $i \in [k]$.
\item[(iii)] Update $\rho^{(n+1)}_{I_n}:=X_n$, $\rho^{(n+1)}_{j}:=\rho^{(n)}_j$ for any $j \in [k] \setminus \{I_n\}$, $E^{(n+1)}_i:=(\rho_i^{(n+1)}, \rho_i')$ for any $i \in [k]$ and $E^{(n+1)}:=\bigcup_{i=1}^k E_i^{(n+1)}$.
\end{enumerate}
\item Align the interval components $(a_n, b_{n}]$, $n \geq 1$, in respective order by identifying $b_{n-1}$ with $a_n$ for $n \geq 2$, and obtain the set $E_{\biguplus}$ equipped with a metric $d_{\biguplus}$ induced by this operation.\end{itemize} 
Then the pair $(E_{\biguplus}, \mu)$ equipped with the metric $d_{\biguplus}$ is an $(1/2, k/2)$-string of beads.
\end{algothm}

\begin{figure}[htb]  \centering \def\svgwidth{\columnwidth} 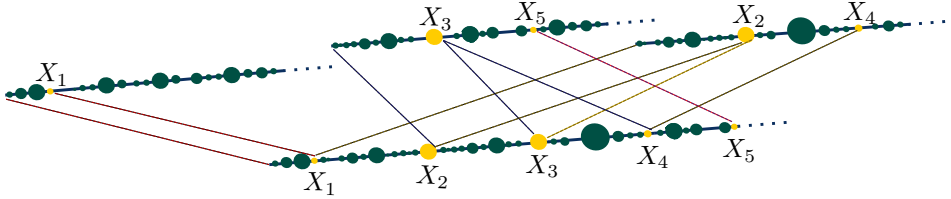 \caption{ Example of merging of strings of beads for $k=3$. First five cut points $X_1, \ldots, X_5$ are displayed.} \label{figure1} \end{figure} 

This algorithm is made mathematically precise and presented for general $(\alpha, \theta_i)$-strings of beads, $\alpha \in (0,1)$ and $\theta_i >0$, $i \in [k]$, in Section \ref{msob}, see Algorithm \ref{algostring} and Theorem \ref{Mainresult}. In general, our merging procedure yields an $(\alpha, \theta:=\sum_{i=1}^k \theta_i)$-string of beads. As a corollary of our result, we recover the following property of PD$(\alpha, \theta)$ vectors, disregarding the regenerative order incorporated by $(\alpha,\theta)$-strings of beads.

\begin{corollary}
Let $\alpha \in (0,1)$ and $\theta_i > 0$, $i \in[k]$. Let $X_i, i \in [k]$, be a sequence of independent random vectors such that $X_i=(X_{i,1}, X_{i,2}, \ldots) \sim {\rm PD}(\alpha, \theta_i), i \in [k]$, and let $Y:=(Y_1,\ldots, Y_k)$ be an independent $k$-dimensional vector $Y$ such that $Y \sim {\rm Dirichlet}(\theta_1, \ldots, \theta_k)$. Then
\begin{equation*}
X:=\left(Y_1 X_{1,1}, Y_2 X_{2,1}, \ldots, Y_k X_{k,1}, Y_1 X_{1,2}, \ldots, Y_k X_{k,2}, Y_1 X_{1,3}, \ldots \right)^{\downarrow} \sim \text{\rm PD}(\alpha,\theta)
\end{equation*}
where $\theta:=\sum_{i=1}^k \theta_i$ and $\downarrow$ denotes the decreasing rearrangement of components.
\end{corollary}

The proof of our main result, Theorem \ref{Mainresult}, goes back to decomposition results for $(\alpha, \theta)$-strings of beads \citep{1} and a stick-breaking construction for $(\alpha, \theta)$-regenerative interval partitions \citep{3}. 

The $(\alpha, \theta)$-tree growth process \citep{1} related to ordered $(\alpha, \theta)$-CRPs generalises Ford's model (i.e. $\theta=1-\alpha$) \citep{12} whose asymptotics were studied earlier by \cite{2}. See also \cite{26} for a multifurcating tree growth process related to Ford's model, and \citep{4} for structural results on regenerative tree growth. 

It was proved in \citep{60} that, for $0<\alpha<1$ and $\theta \geq 0$, the delabelled trees in the $(\alpha, \theta)$-model have a binary fragmentation continuum random tree (CRT) $\mathcal T^{\alpha,\theta}$ as distributional scaling limit, where a two-stage limit was provided earlier in \citep{1}.

The problem of appropriately embedding leaf labels of $(T_n^{\alpha, \theta}, n \geq 1)$ and the reduced trees with edge lengths $\mathcal R_k^{\alpha, \theta}$, arising as scaling limit of the growth process $(T_n^{\alpha, \theta}, n \geq 1)$ reduced to the first $k$ leaves, into $\mathcal T^{\alpha,\theta}$ was solved in \citep{1} by introducing the so-called \textit{$(\alpha,\theta)$-coin tossing construction}.  

Based on our merging operation for strings of beads, Algorithm \ref{algostring}, we develop a branch merging operation (Algorithm \ref{bmalgo}) on Ford CRTs, the class of CRTs $\mathcal T^{\alpha,1-\alpha}$ arising for $\theta=1-\alpha$, $\alpha \in (0,1)$. We can couple $\mathcal T^{\alpha, 1-\alpha}$ and $\mathcal T^{\alpha, 2-\alpha}$ in the sense that we can embed $\mathcal R_k^{\alpha, 2-\alpha}$ into $\mathcal T^{\alpha,2-\alpha}$ given an embedding of $\mathcal R_k^{\alpha,1-\alpha}$ into $\mathcal T^{\alpha,1-\alpha}$, $k\geq 1$. In particular, we construct $\mathcal T^{\alpha, 2-\alpha}$ with leaves embedded from the limiting tree $\mathcal T^{\alpha, 1-\alpha}$ of Ford's model.

In the case when $\alpha =1/2$, and $\theta=1/2$ or $\theta=3/2$, Aldous' Brownian Continuum Random Tree \citep{9,8,7} arises as distributional scaling limit of the delabelled $(\alpha, \theta)$-tree growth process. Leaf labelling is exchangeable in the case when $\alpha=\theta=1/2$, and uniform sampling from the natural mass measure on the Brownian CRT allows to embed leaf labels in this case, using a simplified version of the branch merging algorithm. For the case when $\alpha=1/2$ and $\theta=3/2$ we obtain embedded leaves applying our branch merging operation, yielding a much simpler approach than the coin tossing construction from \citep{1}.

As a by-product of these developments, we obtain the distributional invariance under branch merging of Ford CRTs $\mathcal T^{\alpha, 1-\alpha}$, $\alpha \in (0,1)$, and in particular of the Brownian CRT. Detached from the aim of leaf identification in CRTs, we define the \textit{Branch Merging Markov Chain (BMMC)} using a simplified branch merging operation as the transition rule. The BMMC operates on the space of continuum trees. We prove that, for any $n \in \mathbb N$, a discrete analogue of the BMMC on the space $\mathbb T_n^o$ of rooted unlabelled trees with $n$ leaves and no degree-two vertices (except the root) has a unique stationary distribution supported on the space $\mathbb T_n^{b,o}$ of binary rooted unlabelled trees with $n$ leaves (and no degree-two vertices), to which it converges as time goes to infinity. The Brownian CRT is a stationary distribution of the BMMC, and we conjecture that, for any continuum tree as initial state, the BMMC converges in distribution to the Brownian CRT. 

This article is organized as follows. In Section \ref{Prelim} we recall some background on regenerative composition structures, and give a proper definition of ordered Chinese restaurant processes and $(\alpha, \theta)$-strings of beads. In Section \ref{msob}, we introduce the merging algorithm for strings of beads and state our main result, which is proved in Section \ref{proof}. Branch merging is explained in Section \ref{bm}, where we also recap some background on $\mathbb R$-trees and explain the $(\alpha, \theta)$-tree growth process including its convergence properties.
We use a simplified version of the branch merging algorithm to define the Branch Merging Markov Chain including a discrete analogue, which is shown to have a unique stationary distribution. As an application of general branch merging, we describe leaf embedding for $(\alpha, 2-\alpha)$-tree growth processes in general, and the $(1/2,3/2)$-tree growth process related to the Brownian CRT in particular.

\section{Preliminaries} \label{Prelim}

\subsection{$(\alpha, \theta)$-Chinese restaurant processes and $(\alpha, \theta)$-strings of beads.} \label{crp}

We present some properties of the two-parameter Chinese restaurant process, and refer to \citep{5} for further details. Recall the definition of an $(\alpha, \theta)$-Chinese restaurant process (CRP) $(\Pi_n, n \geq 1)$ for $\alpha \in (0,1)$ and $\theta >0$ from the introduction. The state of the system after $n$ customers have been seated is a random
partition $\Pi_n$ of $[n]$. It is easy to see that, for each particular partition $\pi$ of $[n]$ into $k$ classes of sizes $n_1, \ldots, n_k$, we have
\begin{equation}
\mathbb P(\Pi_n=\pi)=\frac{\prod_{i=1}^{k-1}(\theta+\alpha i)}{[1+\theta]_{n-1}}\prod\limits_{i=1}^k [1-\alpha]_{n_i-1}, \label{partdist}
\end{equation} where, for $x \in \mathbb R$ and $m \in \mathbb N$, we define $[x]_{m}:=x(x+1)(x+2)\cdots(x+m).$ It follows immediately from the fact that the distribution of $\Pi_n$, given by \eqref{partdist}, only depends on block sizes, that the partitions $\Pi_n$ are exchangeable. They are consistent as $n$ varies, and hence induce a \textit{random partition} $\Pi_\infty$ of $\mathbb N$ whose restriction to $[n]$ is $\Pi_n$.

We equip the CRP $(\Pi_n, n \geq 1)$ with a random total order $<$ on the tables, which we call \textit{table order}, see \citep{1}. Independently of the process of seating of customers at tables, we order the tables from left to right according to the following scheme. The second table is put to the right of the first table with probability $\theta/(\alpha+\theta)$, and to the left with probability $\alpha/(\alpha+\theta)$. Conditionally given any of the $k!$ possible orderings of the first $k$ tables, $k \geq 1$, the $(k+1)$st table is put
\begin{itemize}
\item to the left of the first table, or between any two tables with probability $\alpha/(k\alpha+\theta)$ each;
\item to the right of the last table with probability $\theta/(k\alpha+\theta)$.
\end{itemize}

We refer to the CRP with tables ordered according to $<$ as \textit{ordered} CRP, and write $(\widetilde{\Pi}_n, n\geq 1)$ for the process of random partitions of $[n]$ with blocks ordered according to $<$, where we use the convention that $(\Pi_n, n \geq 1)$ orders the blocks according to least labels (\textit{birth order}). For $n \in \mathbb N$, we write 
\begin{equation*}
\Pi_n:=\left(\Pi_{n,1}, \ldots, \Pi_{n,K_n}\right), \qquad \widetilde{\Pi}_n:=\left(\widetilde{\Pi}_{n,1}, \ldots, \widetilde{\Pi}_{n,K_n}\right)
\end{equation*}
for the blocks of the two partitions $\Pi_n$, $\widetilde{\Pi}_n$ of $[n]$, where $K_n$ denotes the number of tables at step $n$. The sizes of these blocks at step $n$ form two \textit{compositions} of $n$, $n \geq 1$, that is, sequences of positive integers $(n_1,\ldots,n_k)$ with sum $n=\sum_{j=1}^k n_j$.  
The theory of CRPs gives us an almost-sure limit for the number of tables $K_n$, i.e. \begin{equation*} L_{\alpha,\theta} = \lim \limits_{n \rightarrow \infty}\frac{K_n}{n^\alpha}, \qquad L_{\alpha,\theta} > 0\quad \text{a.s.}, \end{equation*} as well as limiting proportions $(P_1, P_2, \ldots)$ of costumers at each table in \textit{birth} order, represented as
\begin{equation*} \left(P_1, P_2, P_3, \ldots \right) = \left(M_1, \overline{M}_1 M_2, \overline{M}_1 \overline{M}_2 M_3, \ldots\right), \end{equation*} where the random variables $(M_i, i \geq 1)$ are independent, $M_i \sim \text{\rm Beta}(1-\alpha, \theta+i\alpha)$, and $\overline{M}_i:=1-M_i$. The distribution of ranked limiting proportions is PD$(\alpha, \theta)$. The table sizes in the ordered $(\alpha, \theta)$-CRP $(\widetilde{\Pi}_n, n\geq 1)$ induce a \textit{regenerative composition structure} in the sense of \cite{3}, see also Section \ref{rcs} here.

\begin{definition}[(Regenerative) composition structure]\rm \label{regcs}
A \textit{composition structure} $(\mathcal C_n, n \geq 1)$ is a Markovian sequence of random compositions of $n$, $n=1,2,\ldots$, where the transition probabilities satisfy the property of \textit{sampling consistency}, defined via the following description. Let $n$ identical balls be distributed into an ordered series of boxes according to $\mathcal C_n$, $n \geq 1$, pick a ball uniformly at random and remove it (and delete an empty box if one is created). Then the composition of the remaining $n-1$ balls has the same distribution as $\mathcal C_{n-1}$.

We call a composition structure $(\mathcal C_n, n \geq 1)$ \textit{regenerative} if for all $n > n_1 >1$, given that the first part of $\mathcal C_n$ is $n_1$, the remaining composition of $n-n_1$ has the same distribution as $\mathcal C_{n-n_1}$. \end{definition}

\begin{lemma}[\citep{1}, Proposition 6] \label{cvgcs}
Let $0 <\alpha<1$ and $\theta >0$, and consider an  \textit{ordered} CRP $(\widetilde{\Pi}_n=(\widetilde{\Pi}_{n,1},\ldots, \widetilde{\Pi}_{n,K_n}), n\geq 1)$. For $j \in [n]$, define $S_{n,j}:=\sum _{i=1}^j \lvert \widetilde{\Pi}_{n,i} \rvert$, i.e. $S_{n,j}$ is the number of customers seated at the first $j$ tables from the left. Then,
\begin{equation*}
\left\{\frac{S_{n,j}}{n}, j \geq 0\right\}\rightarrow \mathcal Z_{\alpha, \theta}:=\left\{1-e^{-\xi_t}, t \geq 0\right\}^{\text{\rm cl}} \qquad \text{a.s. as } n \rightarrow \infty,
\end{equation*}
with respect to the Hausdorff metric on closed subsets of $[0,1]$, where $^{\text{\rm cl}}$ denotes the closure in $[0,1]$, and $(\xi_t, t\geq 0)$ is a subordinator with Laplace exponent
\begin{equation*}
\Phi_{\alpha, \theta}(s)=\frac{s\Gamma(s+\theta)\Gamma(1-\alpha)}{\Gamma(s+\theta+1-\alpha)}.\end{equation*}
Furthermore, define $L_n(u):=\#\{j \in [K_n]: S_{n,j}/n \leq u\}$ for $u \in [0,1]$, where for any set $S$, $\#S$ denotes the number of elements in $S$. Then
\begin{equation*}
\lim \limits_{n \rightarrow \infty} \sup \limits_{u \in [0,1]} \lvert n^{-\alpha}L_n(u)-L(u)\rvert =0 \qquad \text{a.s.},\end{equation*}
where $L:=(L(u), u \in [0,1])$ is a continuous local time process for $\mathcal Z_{\alpha,\theta}$ which means that the random set of points of increase of $L$ is $\mathcal Z_{\alpha, \theta}$ a.s..
\end{lemma}

We refer to the collection of open intervals in $[0,1]\setminus \mathcal Z_{\alpha, \theta}$ as the \textit{($\alpha,\theta$)-regenerative interval partition} associated with $(\mathcal C_n, n\geq1)$ and the local time process $L$, where $L(1)=L_{\alpha,\theta}$ a.s.. Note that the joint law of ranked lengths of components of this interval partition is PD$(\alpha,\theta)$. The local time process $(L(u), 0 \leq u \leq 1)$ can be characterized via $\xi$, i.e. we have
\begin{equation*}
L\left(1-e^{-\xi_t}\right)=\Gamma(1-\alpha) \int \limits_{0}^t e^{-\alpha \xi_s}ds, \quad t \geq 0.
\end{equation*}
We consider the inverse local time $L^{-1}$ given by 
\begin{equation}
L^{-1}:[0,L_{\alpha, \theta}] \rightarrow [0,1], \qquad L^{-1}(x):=\inf\{u \in [0,1]: L(u) > x\}.
\end{equation}

Note that $L^{-1}$ is increasing and right-continuous, and hence we can equip the random interval $[0, L_{\alpha, \theta}]$ with the Stieltjes measure $dL^{-1}$. We refer to the pair $([0, L_{\alpha, \theta}],dL^{-1})$ as an $(\alpha, \theta)$-string of beads in the following sense. 

\begin{definition}[String of beads]\rm
A \textit{string of beads} $(I, \mu)$ is an interval $I$ equipped with a discrete mass measure $\mu$. An \textit{$(\alpha, \theta)$-string of beads} is the weighted random interval $([0,L_{\alpha, \theta}], dL^{-1})$ associated with an $(\alpha, \theta)$-regenerative interval partition.
\end{definition}

We equip the space of strings of beads with the Borel $\sigma$-algebra associated with the topology of weak convergence of probability measures on $[0, \infty)$, where the length of a string of beads $(I, \mu)$ is determined by the supremum of the support of the measure $\mu$.

We also use the term $(\alpha, \theta)$-string of beads for measure-preserving isometric copies of the weighted interval $([0, L_{\alpha, \theta}], dL^{-1})$. Since the lengths of the interval components of an $(\alpha, \theta)$-regenerative interval partition are the masses of the atoms of the associated $(\alpha, \theta)$-string of beads, we know that the joint law of ranked masses of the atoms of an $(\alpha, \theta)$-string of beads is PD$(\alpha, \theta)$.

\subsection{Regenerative composition structures}
\label{rcs}
We recap some well-known results for regenerative composition structures from \cite{3}, some of which are based on \cite{40}.
For a (random) closed subset $\mathcal M$ of $[0, \infty]$ let
\begin{equation*}
G(\mathcal M, t):=\sup \mathcal M \cap [0,t], \qquad D(\mathcal M, t):=\inf \mathcal M \cap (t, \infty].
\end{equation*}
A random closed subset $\mathcal M \subset [0, \infty]$ is \textit{regenerative} if for each $t \in [0, \infty)$, conditionally given $\{D(\mathcal M, t) < \infty\}$, the random set $(\mathcal M- D(\mathcal M, t)) \cap [0, \infty]$ has the same distribution as $\mathcal M$ and is independent of $[0, D(\mathcal M, t)] \cap \mathcal M$. 

\begin{theorem}[\citep{3}, Theorem 5.1] \label{cs10}
The closed range $\{S_t, t \geq 0\}^{\text{cl}}$ of a subordinator $(S_t, t\geq 0)$ is a regenerative random subset of $[0,\infty]$. Furthermore, for every regenerative random subset $\mathcal M$ of $[0, \infty]$ there exists a subordinator $(S_t, t\geq 0)$ such that $\mathcal M \,{ \buildrel d \over =}\, \{S_t, t \geq 0\}^{\text{\rm cl}}.$
\end{theorem}	

We call a process $(\widetilde{S}_t, t \geq 0)$ a \textit{multiplicative subordinator} if for $t' >t$, the ratio $(1-\widetilde{S}_{t'})/(1-\widetilde{S}_{t})$ has the same distribution as $(1-\widetilde{S}_{t'-t})$ and is independent of $(\widetilde{S}_u, 0 \leq u \leq t)$. Furthermore, let $\widetilde{\mathcal M}$ denote the closed range of a multiplicative subordinator $(\widetilde{S}_t, t\geq0)$, i.e. $\widetilde{\mathcal M}:=\{ \widetilde{S}_t, t\geq 0\}^{\text cl}$. A multiplicative subordinator $(\widetilde{S}_t, t \geq 0)$ can be obtained from a subordinator $(S_t, t \geq 0)$ via the the mapping from $[0, \infty]$ onto $[0,1]$ defined by $z \mapsto 1-\exp(-z)$, i.e. $\widetilde{S}_t:=1-\exp(-S_t)$. 

For any closed subset $M$ of $[0,1]$ and any $z \in [0,1]$ such that $M \cap (z,1) \neq \emptyset$, we define $M(z)$ as the part of $M$ strictly right to $D(M, z)$, scaled back to $[0,1]$, i.e.
\begin{equation*}
M(z):=\left\{\frac{y-D(M,z)}{1-D(M,z)}: y \in M \cap [D(M, z), 1]\right\}.
\end{equation*}
A random closed subset $\widetilde{\mathcal M} \subset [0,1]$ is called \textit{multiplicatively regenerative} if, for each $z \in [0,1)$, conditionally given $\{D(\widetilde{ \mathcal M}, z)<1\}$ the random set $\widetilde{\mathcal M}(z)$ has the same distribution as $\widetilde{\mathcal M}$ and is independent of $[0,D(\widetilde{\mathcal M},z)] \cap \widetilde{\mathcal M}$.

\begin{proposition}[\citep{3}, Proposition 6.2]\label{cs11}
For two random closed subsets $\widetilde{\mathcal M} \subset [0,1]$ and $\mathcal M \subset [0, \infty]$ related by $\widetilde{\mathcal M}=1-\exp\{-\mathcal M\}$, the random set $\mathcal M$ is regenerative if and only if $\widetilde{\mathcal M}$ is multiplicatively regenerative.
\end{proposition}

We associate with each composition $(n_1,\ldots, n_k)$ of $n$ the finite closed set whose points are partial sums of the parts $n_1,\ldots, n_k$ divided by $n$, i.e. 
\begin{equation}(n_1, \ldots, n_k) \mapsto \{0, n_1/n, (n_1+n_2)/n, \ldots, (n-n_k)/n,1\}=:\widetilde{\mathcal M}_n. \label{raset} \end{equation}
Every regenerative composition structure $(\mathcal C_n, n \geq 1)$ is hence associated with a sequence of random sets $(\widetilde{\mathcal M}_n, n \geq 1)$. Note that each $\widetilde{\mathcal M}_n$ defines an interval partition of $[0,1]$, namely $[0,1] \setminus \widetilde{\mathcal M}_n$, $n \geq 1$.

\begin{lemma}[\citep{3}, Lemma 6.3] \label{crs1}
Let $(\mathcal C_n, n \geq 1)$ be a composition structure and let $(\widetilde{\mathcal M}_n, n \geq 1)$ be the associated sequence of random sets as in $\eqref{raset}$. Then $\widetilde{\mathcal M}_n \rightarrow \widetilde{\mathcal M}$ a.s. as $n \rightarrow \infty$ in the Hausdorff metric for some random closed subset $\widetilde{\mathcal M}$.
\end{lemma}
\begin{lemma}[\citep{3}, Corollary 6.4]
In the setting of Lemma \ref{crs1}, a composition structure $(\mathcal C_n, n \geq 1)$ is regenerative if and only if $\widetilde{\mathcal M}$ is multiplicatively regenerative.
\end{lemma}

For $\alpha \in (0,1)$, we can construct an $(\alpha, 0)$-regenerative interval partition, arising in the limit of the regenerative composition structure induced by an ordered $(\alpha, 0)$-CRP, as the restriction to $[0,1]$ of the range of a stable subordinator of index $\alpha$.

\begin{definition}[Stable subordinator] \rm A subordinator $(S_t, t\geq 0)$ is called \textit{stable} with index $\alpha \in (0,1)$ if it has the \textit{self-similarity} property, that is, for all $t > 0$, $S_t/t^{1/\alpha}\,{\buildrel d \over =}\,S_1.$
\end{definition}

\begin{proposition}[\citep{3}, Section 8.3]\label{cs12}
Let $0<\alpha <1$, and let $\mathcal M_{\alpha}$ be the range of a stable subordinator of index $\alpha>0$. Then, the interval partition of $[0,1]$ generated by $\mathcal M_{\alpha} \cap [0,1]$ is an $(\alpha, 0)$-regenerative interval partition. 
\end{proposition}

\section{Merging strings of beads} \label{msob}

\subsection{The merging algorithm} \label{mergalg}
We first explain the basic structure of our merging operation for strings of beads, given a sequence of cut points $(x_n)_{n \geq 1}$, that is, a sequence of atoms on the strings of beads. We use $\mathbb R$ equipped with the usual distance function as the underlying metric space, and refer to isometric copies of the related intervals when we present applications to continuum trees in Section \ref{bm}.

Let $k \in \mathbb N$, and consider $k$ disjoint strings of beads $(E_i,\mu_i)$, $i \in [k]$, given by 
\begin{equation*} E_i:=(\rho_i,\rho_i' ), \qquad i \in [k], \end{equation*} where $E_i \cap E_j = \emptyset$ for all $i, j \in [k], i \neq j$. Furthermore, suppose that, for every $i \in [k]$, we have a sequence $(\rho_{i,j})_{j \geq 1}$ such that
\begin{equation} E_i= \bigcup\limits_{j=0}^\infty (\rho_{i,j}, \rho_{i,j+1}] \quad\text{and } d(\rho_i, \rho_{i,j}) <d(\rho_i, \rho_{i,j+1}) \text{ for all } j \geq 1, \label{operator} \end{equation} where $\rho_{i,0}:=\rho_i$, $i \in [k]$. We refer to the sequences $(\rho_{i,j})_{j \geq 1}$, $i \in [k]$, as \textit{cut points} for the string of beads $(E_i, \mu_i)$, $i \in[k]$. Define the \textit{set of cut points} $\Upsilon$ by \begin{equation}\Upsilon:=\{\rho_{i,j}, j \geq 1, i \in [k]\},\end{equation} and assume that $\sigma_\Upsilon: \Upsilon \rightarrow \mathbb N$ gives a total order on $\Upsilon$ (i.e. $\sigma_{\Upsilon}$ is bijective) which is consistent with the natural order given by $\mathbb N$ when restricted to $\{\rho_{i,j}, j\geq 1\}$, $i \in [k]$, in the sense that $\sigma_{\Upsilon}(\rho_{i,j}) < \sigma_{\Upsilon}(\rho_{i,j+1})$ for all $j \geq 1$. For $n \in \mathbb N$, we write $x_n:=\sigma_\Upsilon^{-1}(n)$, i.e. $\Upsilon=\{x_n, n\geq 1\}$.

We define the \textit{merged} string of beads $(E_{\biguplus},\mu)$ equipped with the metric $d_{\biguplus}$ by
\begin{equation}E_{\biguplus}:={\biguplus\limits_{i=1}^k}E_i = \biguplus \limits_{i=1}^k (\rho_i, \rho_i') \label{dreidrei} \end{equation} where the operator $\biguplus$ on the intervals $E_i$, $i \in [k]$, the metric $d_{\biguplus}: E_{\biguplus} \times E_{\biguplus} \rightarrow \mathbb R_{0}^+$ and the mass measure $\mu$ on $E_{\biguplus}$ are defined in the following way.
\begin{itemize}

\item For any $n \geq 1$, let $(a_{n}, b_{n}]$ be defined by
\begin{equation}
(a_{n}, b_{n}]:=(\rho_{i,j-1}, \rho_{i,j}]
\end{equation}
where $\rho_{i,j}$ denotes the unique element of $\Upsilon$ with $\sigma_{\Upsilon}(\rho_{i,j})=n$, i.e. $x_{n}=\rho_{i,j}$.

\item Define the set $E_{\biguplus}$ by
\begin{equation}
E_{\biguplus}:=\bigcup \limits_{n=1}^\infty (a_n,b_{n}]=\bigcup_{i=1}^k E_i,
\end{equation}
and equip it with the metric $d_{\biguplus}$ carried forward via these operations, i.e.
\begin{equation}d_{\biguplus} (x,y):=\begin{cases} d(x,b_{n}) + \sum \limits_{k=n+1}^{m-1} d(a_{k}, b_{k}) + d(a_{m}, y),  & x \in (a_{n}, b_{n}], y \in (a_{m}, b_{m}]\\
\lvert d(x, b_{n}) - d(y, b_{n})\rvert, & x,y \in  (a_{n}, b_{n}], n \geq 1,
\label{distance1} \end{cases} \end{equation}
where $n<m$ in the first case. We use the subscript $\biguplus$ to underline that the set $E_{\biguplus}= \bigcup_{i=1}^k E_i$ is equipped with the metric $d_{\biguplus}$.

\item Considering the natural mass measure $\mu$ on the Borel sets $\mathcal B(E_{\biguplus})$ given by
\begin{equation} \mu(A)=\sum \limits_{i=1}^k \mu_i(A \cap E_i), \quad A \in \mathcal B(E_{\biguplus}), \end{equation}
yields a metric space $(E_{\biguplus}, d_{\biguplus})$ endowed with a mass measure $\mu$.\end{itemize}

Note that $(E_{\biguplus}, d_{\biguplus})$ is isometric to an open interval. Considering its completion, we can write $E_{\biguplus}=(\rho, \rho')$ where $\rho:=a_{0}$ and $\rho'$ is the unique element in the completion of $E_{\biguplus}$ such that $d_{\biguplus}(b_{n},\rho')\rightarrow 0$ as $n\rightarrow \infty$. Hence, $(E_{\biguplus}, \mu)$ is a string of beads equipped with the metric $d_{\biguplus}$. Furthermore, note that $d_{\biguplus}(\rho, \rho')=\sum_{i=1}^k d(\rho_i, \rho_i')$.

\begin{remark} \label{closedintervals} Note that we do not necessarily need \textit{open} intervals $(\rho_i, \rho_i')$, $i \in [k]$, in the merging algorithm. In the case when $E_l = [\rho_l, \rho_l')$ for some $l \in [k]$, we interpret our algorithm in the following sense. Consider the first cut point $\rho_{l,1}$ in $E_l$, i.e. $\rho_{l,1}=x_n \text{ such that } x_n \in E_l \text{ and } x_m \notin E_l \text{ for all } 1 \leq m \leq n-1.$
\begin{itemize}
\item If $\rho_{l,1}\neq x_1$, then there is $n \in \mathbb N$, such that $b_{n}=\rho_{l,1}$ and $b_{n-1}=\rho_{i,j}$ for some $j \geq 1$ and $i \in [k] \setminus \{l\}$. In this case, we consider the interval component $(a_n, b_{n}]=(\rho_l, \rho_{l,1}]$ in our algorithm, and set $\mu(b_{n-1}):=\mu_i(\rho_{i,j})+\mu_l(\rho_l)$. 
\item If $\rho_{l,1}=x_1$, then $[a_1, b_1]=[\rho_l, x_1]$ and we obtain an interval $[\rho=\rho_l, \rho')$ as a result of the merging procedure with $\mu(a_1):=\mu_l(\rho_l)$.
\end{itemize}
If $E_l=(\rho_l, \rho_l']$, we might have that $x_n=\rho_l'$ for some $n \in \mathbb N$. Then the interval $E_l$ is split into only finitely many components. If $E_l=[\rho_l, \rho_l']$ is a closed interval, we naturally combine the described conventions.
\end{remark}

\subsection{Cut point sampling and main result}
We now assume that our strings of beads are rescaled $(\alpha, \theta_i)$-string of beads for $\alpha \in (0,1)$ and $\theta_i >0$, $ i \in [k]$. More precisely, let $E_i=(\rho_i, \rho_i')$, $i\in [k]$, and let $\mu$ be a mass measure on $E=\bigcup_{i=1}^k E_i$ such that
\begin{equation}\left(\mu(E_1),  \ldots, \mu(E_k)\right)\sim \text{{\rm Dirichlet}}(\theta_1,\ldots, \theta_k). \label{ms0} \end{equation} Furthermore, suppose that
\begin{equation*}
\left(\mu(E_i)^{-\alpha}E_i, \mu(E_i)^{-1}\mu \restriction_{E_i}\right), \qquad i \in [k],
\end{equation*}
are $(\alpha,\theta_i)$-strings of beads, respectively, independent of each other and of \eqref{ms0}. We construct an $(\alpha, \theta=\sum_{i=1}^k \theta_i)$-string of beads from $(E_i, \mu\restriction_{E_i} ), i \in [k]$, using the merging algorithm presented in Section \ref{mergalg}. The procedure is based on sampling an atom of an $(\alpha, \theta)$-string of beads such that the induced mass split has a {\rm Dirichlet}$(\alpha, 1-\alpha, \theta)$ distribution. \cite{1} describe a sampling procedure for such an atom, and show that the atom splits the $(\alpha, \theta)$-string of beads into a rescaled independent $(\alpha, \alpha)$- and a rescaled independent $(\alpha, \theta)$-string of beads.

\begin{proposition}[\citep{1}, Proposition 10/14(b), Corollary 15] \label{cointoss} Consider an $(\alpha, \theta)$-string of beads $(I, \mu):=([0, L_{\alpha, \theta}], dL^{-1})$ for some $\alpha \in (0,1)$ and $\theta> 0$ with associated $(\alpha, \theta)$-regenerative interval partition $[0,1] \setminus \mathcal Z_{\alpha,\theta}$, where $\mathcal Z_{\alpha, \theta}=\{1-\exp(-\xi_t), t \geq 0\}^{\text{cl}}$, cf. Lemma \ref{cvgcs}. Define a switching probability function by
\begin{equation*} p:[0,1] \rightarrow [0,1], \quad u \mapsto \frac{(1-u)\theta}{(1-u)\theta+u\alpha}.
\end{equation*}
For $t \geq 0$, select a block $(1-e^{-\xi_{t^-}}, 1- e^{-\xi_t})$ of $[0,1] \setminus \mathcal Z_{\alpha,\theta}$ with $\Delta \xi_t:=\xi_{t}- \xi_{t^-} >0$ with probability 
\begin{equation*} p\left(e^{-\Delta\xi_t}\right) \cdot \prod_{s < t} \left(1-e^{-\Delta \xi_s}\right) \end{equation*}
and consider the local time $ X:=L(1-e^{-\xi_t})$ at the selected block $(1-e^{-\xi_{t^-}}, 1- e^{-\xi_t})$, which we call \textit{switching time}.
Then the following random variables are independent.
\begin{itemize}
\item The mass split
\begin{equation} \left( \mu([0, X)), \mu(X), \mu((X, L_{\alpha, \theta}])\right) \sim \text{{\rm Dirichlet}}\left(\alpha, 1- \alpha, \theta\right); \label{splitsob0} \end{equation}
\item the $(\alpha, \alpha)$-string of beads
\begin{equation}
\left( \mu\left([0,X)\right)^{-\alpha} [0,X), \mu\left([0,X)\right)^{-1} \mu \restriction_{[0,X)} \right); \label{splitsob1}
\end{equation}
\item the $(\alpha, \theta)$-string of beads
\begin{equation}
\left( \mu\left((X, L_{\alpha, \theta}]\right)^{-\alpha} (X, L_{\alpha, \theta}], \mu\left((X, L_{\alpha, \theta}]\right)^{-1} \mu \restriction_{(X, L_{\alpha, \theta}]} \right). \label{splitsob2}
\end{equation}
\end{itemize}
\end{proposition}

We refer to the sampling procedure described in Proposition \ref{cointoss} as \textit{$(\alpha, \theta)$-coin tossing sampling} in the string of beads $([0, L_{\alpha, \theta}], dL^{-1})$. Informally speaking, we can visualise this by a walker starting in $0$, walking up the string of beads $([0, L_{\alpha, \theta}], dL^{-1})$, tossing a coin for each of the (infinite number of atom) masses where the probability for heads is given by $p(\exp(-\Delta \xi_t))$, and hence depends on the relative remaining mass after an atom; let the walker stop the first time the tossed coin shows heads and select the related atom. 

\begin{algorithm}[Construction of an $(\alpha, \theta)$-string of beads] \label{algostring}\rm  Let $\alpha \in (0,1)$ and $\theta_1, \ldots, \theta_k >0$.
\begin{itemize}
 \item\textit{Input}. A mass measure $\mu$ on $E:=\bigcup_{i=1}^k E_i$, $E_i=(\rho_i, \rho_i')$, $i\in [k]$, such that 
\begin{equation}\left(\mu(E_1), \ldots, \mu(E_k)\right)\sim \text{{\rm Dirichlet}}(\theta_1, \ldots, \theta_k) \label{ms00} \end{equation}
where the pairs $(\mu(E_i)^{-\alpha}E_i, \mu(E_i)^{-1}\mu\restriction_{E_i})$, $i \in [k]$, are $(\alpha,\theta_i)$-strings of beads, respectively, independent of each other and of the mass split \eqref{ms00}.

\item\textit{Output}. A sequence of cut points $(X_n)_{n \geq 1}$ and a sequence of interval components $\{(a_n, b_{n}]\}_{n \geq 1}$ of $E_i, i \in [k]$.
\end{itemize}
\textit{Initialisation}. $n=1$, $E^{(1)}_i:= E_i$, $i \in [k]$, and $E^{(1)}:=\bigcup_{i \in [k]}E_i$, $\rho_i^{(0)}:=\rho_i$,  $i \in [k]$. \\

For $n \geq 1$, given $(E, \mu)$ and the previous steps of the algorithm, do the following.
\begin{enumerate}
\item[(i)] Select one of the intervals $E_i^{(n)}$, $i \in [k]$, proportionally to mass, i.e. let $I_n$ be a random variable taking values in $\{1,\ldots,k\}$ such that, for $i \in [k]$, $I_n=i$ with probability $\mu(E_i^{(n)})/\mu(E^{(n)})$. 

Conditionally given $I_n$, pick an atom $X_n$ of the interval $E_{I_n}^{(n)}$ according to $(\alpha, \theta_{I_n})$-coin tossing sampling in the string of beads
\begin{equation*}
\left(\mu\left(E_{I_n}^{(n)}\right)^{-\alpha} E_{I_n}^{(n)}, \mu\left(E_{I_n}^{(n)}\right)^{-1} \mu \restriction_{E_{I_n}^{(n)}}\right).
\end{equation*}

\item[(ii)] Define the interval $(a_n, b_{n}]$ by
\begin{equation}
(a_n, b_{n}]:=(\rho_{I_n}^{(n-1)},X_n].\label{enotation}
\end{equation}

\item[(iii)] Update $\rho_{I_n}^{(n)}:=X_n$ and $\rho_i^{(n)} := \rho_i^{(n-1)}$ for $i \in [k]\setminus \{I_n\}$, i.e.
\begin{equation}
E_{I_n}^{(n+1)}:=(X_n, \rho_{I_n}'), \quad E_i^{(n+1)} := E_i^{(n)} \text{ for } i \in [k] \setminus \{I_n\}, \label{notation} \end{equation} 
and set $E^{(n)}=\bigcup_{i=1}^k E_i^{(n)}$.
\end{enumerate}
\end{algorithm}

We define $E_{\biguplus}$ by aligning the intervals $(a_n, b_{n}]$ in respective order. More precisely, we consider $\Upsilon:=\{X_n, n \geq 1\}$, $\sigma_\Upsilon: \Upsilon \rightarrow \mathbb N$, $\sigma({X_n}):=n$, and define
\begin{equation*} E_{\biguplus}:=\biguplus_{n\geq 1} (a_n, b_{n}], \end{equation*}
equipped with the metric $d_{\biguplus}$, where the operator $\biguplus$ was defined in \eqref{dreidrei}-\eqref{distance1}. We can write $E_{\biguplus}=(\rho,\rho')$. Informally, Algorithm \ref{algostring} selects one of the strings of beads proportionally to mass, and then chooses an atom on the selected string according to $(\alpha, \theta)$-coin tossing sampling. In this string of beads, the part  before the selected atom is cut off and the algorithm proceeds with the remainder of this string of beads and all the other strings of beads which stay unchanged. A new string of beads is then obtained by concatenating the parts in the order in which they are cut off in the algorithm.

Our main result here is the following.

\begin{theorem}[Merging of $(\alpha, \theta)$-strings of beads] \label{Mainresult}
Let $\alpha \in (0,1)$ and $\theta_1, \ldots, \theta_k >0$, $k \in \mathbb N$. Let $E_1, \ldots, E_k$ be $k$ disjoint intervals and $\mu$ a mass measure on $E:=\bigcup_{i=1}^k E_i$ such that 
\begin{equation*} \left(\mu(E_1), \ldots, \mu(E_k)\right) \sim {{\rm Dirichlet}}\left(\theta_1, \ldots, \theta_k\right).
\end{equation*}
Furthermore, suppose that the rescaled pairs 
\begin{equation*} \left(\mu(E_i)^{-\alpha}E_i, \mu(E_i)^{-1} \mu \restriction_{E_i}\right) \end{equation*} 
are $(\alpha, \theta_i)$-string of beads, $i \in [k]$, respectively, independent of each other and of the {\rm Dirichlet} mass split $(\mu(E_1), \ldots, \mu(E_k))$. Let $E_{\biguplus}:=\biguplus_{i=1}^k E_i$ be constructed as in Algorithm $\ref{algostring}$ with associated distance $d_{\biguplus}$ and mass measure $\mu$. Then the pair $(E_{\biguplus}, \mu)$ is an $(\alpha, \theta)$-string of beads with $\theta:=\sum_{i=1}^k \theta_i$.
\end{theorem}

\subsection{Proof of the main result} \label{proof}
We first recall and establish some preliminary results, including decomposition and construction rules for $(\alpha,\theta)$-regenerative interval partitions and $(\alpha, \theta)$-strings of beads.

\subsubsection{Some preliminary results}
$(\alpha,\theta)$-regenerative interval partition can be constructed as follows.

\begin{proposition}[\citep{1}, Corollary 8]\label{intervpart}
Let $0 < \alpha <1$ and $\theta \geq 0$, and consider a random vector $(G, D-G, 1-D) \sim \text{{\rm Dirichlet}}(\alpha, 1-\alpha, \theta)$. Conditionally given $(G,D)$, construct an interval partition of $[0,1]$ as follows. Let $(G,D)$ be one component, and construct the interval components in $[0, G]$ and in $[D,1]$ by linear scaling of an independent $(\alpha, \alpha)$- and an independent $(\alpha, \theta)$-regenerative interval partition, respectively. The interval partition obtained in this way is an $(\alpha, \theta)$-regenerative interval partition.
\end{proposition}

\begin{remark}[The special case $\theta=0$] \label{theta0}\rm
The case $\theta=0$ is considered in \citep{23}, Proposition 15. In this case, $G \sim \text{\rm Beta}(\alpha,1-\alpha)$ and $(G,1)$ is the last interval component of $[0,1] \setminus \mathcal Z_{\alpha, 0}$, and the restriction of $\mathcal Z_{\alpha, 0}$ to $[0,G]$ is a rescaled copy of $\mathcal Z_{\alpha,\alpha}$, where $\mathcal Z_{\alpha, \theta}$, $\alpha \in (0,1)$, $\theta \geq 0$, was defined in Lemma \ref{cvgcs}.
\end{remark}

We can formulate an analogous version of Proposition \ref{intervpart} for $(\alpha, \theta)$-strings of beads, $0 < \alpha < 1$ and $\theta \geq 0$, which is based upon \citep{1}, Proposition 14(b).

\begin{proposition}[\citep{1}, Proposition 14(b)] \label{sobconstruc}
Let $0 < \alpha <1$ and $\theta \geq 0$, and consider a random vector $(G, D-G, 1-D) \sim \text{{\rm Dirichlet}}(\alpha, 1-\alpha, \theta)$. Conditionally given $(G,D)$, let $(I_1=[0, L_{\alpha, \alpha}], \mu_1)$ be an independent $(\alpha, \alpha)$-string of beads and $(I_2=[0, L_{\alpha, \theta}], \mu_2)$ an independent $(\alpha, \theta)$-string of beads. Define \begin{equation*}
L_{\alpha, \theta}':=G^{\alpha}L_{\alpha, \alpha}+(1-D)^{\alpha}L_{\alpha, \theta},\end{equation*} the interval $I:=[0,L_{\alpha, \theta}']$ and the mass measure $\mu$ on $I$ by
\begin{equation*} \mu(y) = \begin{cases} G \cdot \mu_1(G^{-\alpha}  y) & \text{if } y \in [0,G^{\alpha}L_{\alpha, \alpha}),\\
D-G & \text{if } y=G^{\alpha}L_{\alpha,\alpha}, \\
 (1-D)\cdot \mu_2((1-D)^{-\alpha}(y-G^{\alpha}L_{\alpha,\alpha}))  & \text{if } y \in (G^{\alpha}L_{\alpha,\alpha},L_{\alpha, \theta}']. \end{cases} \end{equation*}
Then $(I, \mu)$ is an $(\alpha, \theta)$-string of beads.
\end{proposition}

The following theorem shows how an $(\alpha, \theta)$-regenerative interval partition, $\alpha \in(0,1)$ and  $\theta > 0$, can be constructed via a stick-breaking scheme, and a sequence of i.i.d. stable subordinators of index $\alpha$.

\begin{proposition}[\cite{3}, Theorem 8.3] \label{intpartstick}
Let $(Y_i)_{i \geq 1}$ be i.i.d. with $B_1 \sim  {\rm Beta}(1, \theta)$ for some $\theta >0$, and define the sequence $(V_n)_{n \geq 1}$ by 
\begin{equation}
V_n:= 1- \prod \limits_{i=1}^n (1-Y_i), \qquad n=1,2, \ldots. \label{vaus}
\end{equation}
For $\alpha \in (0,1)$, let $\mathcal M_{\alpha}(n)$, $n \geq 1$, be independent copies of the range $\mathcal M_{\alpha}$ of a stable subordinator of index $\alpha$, and define a random closed subset $\widetilde{\mathcal M}(\alpha,\theta) \subset [0,1]$ by
\begin{equation}
\widetilde{\mathcal  M}(\alpha, \theta):=\{1\} \cup \bigcup \limits_{n=1}^\infty \left([V_{n-1}, V_n] \cap \left(V_{n-1}+\mathcal M_{\alpha}(n)\right) \right) \label{sets1}
\end{equation} where $V_0:=0$. Then, $\widetilde{\mathcal M}(\alpha, \theta)$ is a multiplicatively regenerative random subset of $[0,1]$ which can be represented as $\widetilde{\mathcal M}(\alpha,\theta) = 1-\exp(-\mathcal M(\alpha,\theta))$, where $\mathcal M(\alpha,\theta)$ is the range of a subordinator with Laplace exponent
\begin{equation*}
\Phi_{\alpha, \theta}(s)=\frac{s\Gamma(s+\theta)\Gamma(1-\alpha)}{\Gamma(s+\theta+1-\alpha)}.  \end{equation*} 
\end{proposition}

\begin{remark}[Sliced splitting] \rm Note that $[0,1]\setminus \widetilde{\mathcal M}({\alpha,\theta})$ is an $(\alpha,\theta)$-regenerative interval partition. \cite{22} refers to the method presented in Proposition \ref{intpartstick} as \textit{sliced splitting}: first split the interval $[0,1]$ according to the stick-breaking scheme with Beta$(1,\theta)$ variables, $\theta >0$, to obtain the stick-breaking points $(V_n)_{n\geq 1}$. Then, for each $n \geq 1$, use an independent copy of a regenerative set derived from a stable subordinator with index $\alpha \in (0,1)$ to split the interval $(V_{n-1}, V_n)$, $n \geq 1$, $V_0:=0$. In other words, we first split the interval $[0,1]$ according to a $(0,\theta)$-regenerative interval partition with $\theta >0$, and then shatter each part according to an $(\alpha, 0)$-regenerative interval partition, $\alpha \in (0,1)$.
\end{remark}

As a consequence of Proposition \ref{intpartstick}, an $(\alpha, \theta)$-regenerative interval partition can be constructed via independent $(\alpha, 0)$-regenerative interval partitions.

\begin{corollary}
\label{intpartstick2}
Let $(Y_i)_{i \geq 1}$ be i.i.d. with $Y_1 \sim {\rm Beta}(1, \theta)$ for some $\theta >0$, and let $(V_n)_{n \geq 1}$ be defined as in \eqref{vaus}. For $\alpha\in (0,1)$, let $\mathcal M^{*}_{\alpha}(n)$, $n\geq 1$, be independent copies of the random closed subset  $\mathcal M^{*}_{\alpha}$ of $[0,1]$ associated with an $(\alpha, 0)$-regenerative interval partition of $[0,1]$ given by $[0,1] \setminus \mathcal M^{*}_{\alpha}$.
Define the random set $\widetilde{\mathcal M}^*(\alpha,\theta)$ by \begin{equation} \widetilde{\mathcal  M}^*(\alpha, \theta):=\{1\} \cup \bigcup \limits_{n=1}^\infty \left(V_{n-1}+\left(V_{n} - V_{n-1}\right) \mathcal M^{*}_{\alpha}(n)\right). \label{alphathetaaa}\end{equation} Then $\widetilde{\mathcal M}^*(\alpha, \theta)$ is a multiplicatively regenerative random subset of $[0,1]$, and hence  defines an $(\alpha, \theta)$-regenerative interval partition of $[0,1]$ via $[0,1] \setminus \widetilde{\mathcal M}^*(\alpha, \theta)$.
\end{corollary}
\begin{proof}
We apply Proposition \ref{intpartstick}, i.e. we show that the sets $V_{n-1}+(V_n-V_{n-1})\mathcal M_{\alpha}^{*}(i)$ are of the form $[V_{n-1},V_n]\cap\left(V_{n-1}+\mathcal M_{\alpha}(n)\right)$ as in $\eqref{sets1}$. Consider the closure $\mathcal M_\alpha$ of the range of a stable subordinator $S_\alpha$ of index $\alpha \in(0,1)$ as defined in Proposition \ref{intpartstick}, i.e. $\mathcal M_\alpha:=\{S_\alpha(t), t \geq 0\}^{\text{cl}}$.
Since $S_\alpha$ is stable, we obtain, for any $c >0$, \begin{equation*} \{S_\alpha(t), t \geq 0\}^{\text{cl}}=\{S_\alpha(tc), t \geq 0\}^{\text{cl}} \,{\buildrel d \over =}\,\{c^{1/\alpha}S_\alpha(t), t \geq 0\}^{\text{cl}} = c^{1/\alpha}\{S_\alpha(t), t \geq 0\}^{\text{cl}}. \end{equation*}
Therefore, for any $r > 0$, $r \mathcal M_{\alpha} \,{\buildrel d \over =}\,\mathcal M_{\alpha}$, and hence, for any $r >0$, we have 
\begin{equation*} [0,r] \cap \mathcal M_\alpha\,{\buildrel d \over =}\, r \cdot \left( [0,1] \cap \mathcal M_\alpha\right).\end{equation*} 
Furthermore, by stationary and independent increments of subordinators, and the independence of $(V_n)_{n \geq 1}$ and $\mathcal M_{\alpha}(n)$, $n\geq1$, in Proposition \ref{intpartstick}, we obtain
\begin{equation*} [V_{n-1}, V_n] \cap \left(V_{n-1} +\mathcal M_{\alpha}(n)\right) \,{\buildrel d \over =}\, V_{n-1}+ [0, V_n-V_{n-1}] \cap \mathcal M_{\alpha}(n). \end{equation*} 

Hence, by Proposition \ref{cs12}, the sets $[0,V_{n}- V_{n-1}] \cap  \mathcal M_\alpha(n)$, $n \in \mathbb N$, define independent $(\alpha,0)$-regenerative interval partitions of $[0,1]$ rescaled by the factor $(V_n-V_{n-1})$. More precisely, we use that if $\mathcal V$ is a random interval partition or closed subset of $[0, \infty]$ such that, for any $r > 0$, $r\mathcal V\,{\buildrel d \over =}\,\mathcal V$ and $W$ is an independent random variable such that $0 < W < \infty$ a.s., then $W \mathcal V \,{\buildrel d \over =}\, \mathcal V$. 

Consequently, the sets $[0,V_n-V_{n-1}]\cap \mathcal M_{\alpha}(n)$, $n\geq 1$, are independent $(\alpha, 0)$-regenerative interval partitions rescaled by $(V_n-V_{n-1})$, and hence each set $[0,V_n-V_{n-1}]\cap \mathcal M_{\alpha}(n)$, $n\geq 1$, has the same distribution as $\left(V_{n} - V_{n-1}\right) \mathcal M^{*}_{\alpha}(n)$.
By Proposition \ref{intpartstick}, we conclude that $\widetilde{\mathcal M}^*(\alpha, \theta)$ defines an $(\alpha, \theta)$-regenerative interval partition.
\end{proof}

\subsubsection{Strings of beads: splitting and merging}

The proof of Theorem \ref{Mainresult} uses an induction on the splitting steps $n \geq 1$. In Lemma \ref{kone} we describe the initial step. Its proof requires some basic properties of the {\rm Dirichlet} distribution.

\begin{proposition} \label{dirprop} Let $X:=(X_1, \ldots, X_k) \sim {\rm Dirichlet}(\theta_1, \ldots, \theta_k)$ for some $\theta_i >0$, $i \in [k]$, and $k \in \mathbb N$. 
\begin{enumerate}
\item[(i)] \textit{Aggregation}. For $i,j \in [k]$ with $i < j$ define  \begin{equation*} X':=(X_1,\ldots, X_{i-1}, X_i+X_j, X_{i+1},\ldots,X_{j-1}, X_{j+1},\ldots, X_k).\end{equation*} Then, $X' \sim {\rm Dirichlet}\left(\theta_1, \ldots, \theta_{i-1}, \theta_i+\theta_j, \theta_{i+1}, \ldots, \theta_{j-1}, \theta_{j+1}, \ldots, \theta_k\right). $
\item[(ii)] \textit{Decimation}. Let $\alpha_1, \alpha_2, \alpha_3 \in (0,1)$ be such that $\alpha_1+\alpha_2+\alpha_3=1$. Fix $i \in [k]$, and consider a random vector $(P_1, P_2, P_3) \sim {\rm Dirichlet}(\alpha_1\theta_i, \alpha_2\theta_i, \alpha_3\theta_i)$ which is independent of $X$. Then \begin{equation*} X'':=(X_1, \ldots, X_{i-1}, P_1X_i, P_2X_i, P_3X_i, X_{i+1}, \ldots, X_k).\end{equation*} 
has a $\text{{\rm Dirichlet}}\left(\theta_1, \ldots, \theta_{i-1}, \alpha_1\theta_i, \alpha_2 \theta_i, \alpha_3\theta_i, \theta_{i+1}, \ldots, \theta_k\right)$ distribution.

\item[(iii)] \textit{Size-bias}. Let $I$ be an index chosen such that $\mathbb P(I=i|(X_1, \ldots, X_k))=X_i$ a.s. for $i \in [k]$. Then, for any $i \in [k]$, conditionally given $I=i$,
\begin{equation*}
X=(X_1, \ldots, X_k) \sim {\rm Dirichlet}\left(\theta_1, \ldots, \theta_{i-1}, \theta_{i}+1, \theta_{i+1},\ldots, \theta_k\right).
\end{equation*}
\item[(iv)] \textit{Two-dimensional marginals}. For any $i,j \in [k]$ with $i \neq j$, 
 \begin{equation*} X_i/(X_i+X_j) \sim {\rm Beta}(\theta_i, \theta_j).\end{equation*}
\item[(v)] \textit{Deletion}. For any $i \in [k]$, the vector 
\begin{equation*}
X^*:=\left( {X_1}/{(1-X_i)}, \ldots, {X_{i-1}}/{(1-X_i)}, {X_{i+1}}/{(1-X_i)}, \ldots, {X_k}/{(1-X_i)}\right)
\end{equation*}
has a ${\rm Dirichlet}(\theta_1, \ldots, \theta_{i-1}, \theta_{i+1}, \ldots, \theta_k)$ distribution.
\end{enumerate}
\end{proposition}
\begin{proof}
(i) and (ii) can be found in \citep{42} as Proposition 13 and Proposition 14/Remark 15, for instance. (iii) is Lemma 17 in \citep{43}. (iv) and (v) follow directly from the representation of the {\rm Dirichlet} distribution via independent Gamma variables.
\end{proof}

\begin{lemma}\label{kone}
In Algorithm \ref{algostring}, the following random variables are independent.
\begin{itemize} 
\item The mass $\mu((\rho_{I_1}, X_1])\sim {\rm Beta}(1, \theta)$;
\item the $(\alpha, 0)$-string of beads \begin{equation*}  \left(\mu((\rho_{I_1}, X_1])^{-\alpha}(\rho_{I_1}, X_1], \mu((\rho_{I_1}, X_1])^{-1}\mu\restriction_{(\rho_{I_1}, X_1]}\right); \end{equation*}  
\item the $(\alpha, \theta_i)$-strings of beads \begin{equation*}  \left(\mu(E_i^{(2)})^{-\alpha}E_i^{(2)}, \mu(E_i^{(2)})^{-1}\mu \restriction_{E_i^{(2)}}\right), i \in[k]; \end{equation*}  
\item the mass split \begin{equation*}  \left(\mu(E_1^{(2)})/\mu(E^{(2)}), \ldots, \mu(E_k^{(2)})/\mu(E^{(2)})\right) \sim \text{{\rm Dirichlet}}(\theta_1,\ldots, \theta_k). \end{equation*}  
\end{itemize}
\end{lemma}

The proof of Lemma \ref{kone} will use the following proposition, which is elementary.

\begin{proposition} \label{indie} Let $I$ be a random variable taking values in a countable set $\mathbb I$. Futhermore, consider a random variable $Z$ taking values in some measurable space $(\mathbb V, \mathcal Z)$ whose distribution conditional on $I=i$ does not depend on $i \in \mathbb \mathbb I$. Then, $Z$ is independent of $I$, and for any other random variable $Z'$ taking values in some measurable space $(\mathbb V', \mathcal Z')$ that is conditionally independent of $Z$ given $I$, the pair $(Z', I)$ is independent of $Z$.
\end{proposition}
\begin{proof}[Proof of Lemma \ref{kone}] Conditionally given $I_1$, i.e. given that the interval $E_{I_1}$ is selected in the first step of Algorithm \ref{algostring}, by Proposition \ref{dirprop}(iii), we have 
\begin{align*}
&\left(\mu(E_1), \ldots, \mu(E_{I_1-1}), \mu(E_{I_1}), \mu(E_{I_1+1}), \ldots, \mu(E_k)\right) \\ & \hspace{2cm} \sim {\rm Dirichlet}\left(\theta_1, \ldots, \theta_{I_1-1}, \theta_{I_1}+1, \theta_{I_1+1}, \ldots, \theta_{I_k}\right).
\end{align*} Note that, conditionally given $I_1$, $X_1$ is an element of ${E_{I_1}}$ which is independent of $E_i, i \in [k] \setminus \{I_1\}$. We apply $(\alpha, \theta_{I_1})$-coin tossing sampling in the $(\alpha, \theta_{I_1})$-string of beads $(\mu(E_{I_1})^{-\alpha}E_{I_1}, \mu(E_{I_1})^{-1}\mu \restriction_{E_{I_1}})$, which gives an atom $X_1$ splitting $E_{I_1}$ according to {\rm Dirichlet}$(\alpha, 1-\alpha, \theta_{I_1})$, i.e. \begin{equation} {\mu(E_{I_1})}^{-1}\left({\mu((\rho_{I_1},X_1))}, {\mu(X_1)}, {\mu((X_1,\rho'_{I_1}))}\right)\sim \text{\text{{\rm Dirichlet}}}(\alpha, 1-\alpha, \theta_{I_1}),\label{ms1} \end{equation}
see Proposition \ref{cointoss}. Proposition \ref{dirprop}(ii), with $i=I_1$, $P_1=\mu((\rho_{I_1}, X_1))/\mu(E_{I_1})$, $P_2=\mu(X_1)/\mu(E_{I_1})$,  $P_3=\mu(( X_1, \rho_{I_1}'))/\mu(E_{I_1})$, and $\alpha_1 = \alpha/(\theta_{I_1}+1)$, $\alpha_2 =(1-\alpha)/(\theta_{I_1}+1)$, $\alpha_3=\theta_{I_1}/(\theta_{I_1}+1)$, yields that the distribution of \begin{equation}
\left(\mu\left((\rho_{I_1},X_1)\right), \mu(X_1), \mu(E_1), \ldots, \mu(E_{I_1-1}), \mu((X_1, \rho_{I_1}')), \mu(E_{I_1+1}),\ldots, \mu(E_k)\right)\label{ms2} \end{equation} is $\text{\text{{\rm Dirichlet}}}(\alpha, 1-\alpha, \theta_{1},\ldots, \theta_{I_1-1}, \theta_{I_1}, \theta_{I_1+1}, \ldots, \theta_{k})$. By Proposition \ref{cointoss}, the pairs
\begin{equation}
\left(\mu((\rho_{I_1}, X_1))^{-\alpha}(\rho_{I_1}, X_1), \mu((\rho_{I_1}, X_1))^{-1} \mu \restriction_{(\rho_{I_1}, X_1)}\right) \label{bead1} \end{equation} and \begin{equation}\left(\mu((X_1, \rho_{I_1}'))^{-\alpha}(X_1, \rho_{I_1}'), \mu((X_1, \rho_{I_1}'))^{-1}\mu \restriction_{(X_1, \rho_{I_1}')}\right) \label{bead2} \end{equation}
are $(\alpha, \alpha)$- and $(\alpha, \theta_{I_1})$-strings of beads, respectively, independent of each other and the mass split \eqref{ms1}. 

By Proposition \ref{dirprop}(iv) and \eqref{ms2} we have that $\mu((\rho_{I_1}, X_1))/\mu((\rho_{I_1}, X_1])\sim \text{\rm Beta}(\alpha, 1-\alpha)$. Furthermore, by Proposition \ref{dirprop}(i), $\mu((\rho_{I_1}, X_1]) \sim \text{\rm Beta}(1, \theta)$
with $\theta=\sum_{i=1}^k \theta_i$. By Proposition \ref{intervpart}, Remark \ref{theta0} and Proposition  \ref{sobconstruc}, we conclude that \begin{equation*} (\mu((\rho_{I_1},X_1])^{-\alpha}(\rho_{I_1},X_1], \mu((\rho_{I_1},X_1])^{-1}\mu \restriction_{(\rho_{I_1},X_1]}) \end{equation*} is an $(\alpha,0)$-string of beads. The distribution of  \begin{equation*} (\mu(E_1^{(2)})/\mu(E^{(2)}), \mu(E_2^{(2)})/\mu(E^{(2)}), \ldots, \mu(E_k^{(2)})/\mu(E^{(2)})) \end{equation*} follows directly from Proposition \ref{dirprop}(v).

The strings of beads
\begin{equation*}  \left(\mu(E_i^{(2)})^{-\alpha}E_i^{(2)}, \mu(E_i^{(2)})^{-1}\mu \restriction_{E_i^{(2)}}\right), i \in[k],  \end{equation*}  
are independent $(\alpha, \theta_i)$-strings of beads, as, for $i \in [k] \setminus \{I_1\}$, these are not affected by the algorithm, conditionally given $I_1$, and $E_{I_1}^{(2)}=(X_1, \rho_{I_1}')$.

To deduce the claimed independence, note that the distributions of the four random variables mentioned in Lemma \ref{kone} do not depend on $I_1$, i.e. we can apply Proposition \ref{indie} to these variables inductively with $I=I_1$.
\end{proof}

We are now ready to prove our main result, Theorem \ref{Mainresult}.

\begin{proof}[Proof of Theorem \ref{Mainresult}]
We use Corollary \ref{intpartstick2} to show the claim. Let $(X_n)_{n \geq 1}$ be a sequence of random atoms of $E$, as sampled in Step (i) of Algorithm \ref{algostring}. 

By Lemma \ref{kone}, the distribution of the mass split obtained via $X_1$ \begin{equation*} \left(\mu((\rho_{I_1},X_1)), \mu(X_1), \mu(E_1), \ldots, \mu(E_{I_1-1}), \mu((X_1, \rho_{I_1}')), \mu(E_{I_1+1}),\ldots, \mu(E_k)\right) \end{equation*} is {{\rm Dirichlet}}$(\alpha, 1-\alpha, \theta_{1},\ldots, \theta_{I_1-1}, \theta_{I_1}, \theta_{I_1+1}, \ldots, \theta_{k})$. Again by Lemma \ref{kone}, 
\begin{equation*}
Y_1:=\mu((\rho_{I_1}, X_1]) \sim \text{\rm Beta}(1, \theta),
\end{equation*}
where $Y_1$ is independent of the $(\alpha,\theta_i)$-strings of beads
\begin{equation}
\left(\mu(E_i^{(2)})^{-\alpha}E_i^{(2)}, \mu(E_i^{(2)})^{-1}\mu \restriction_{E_i^{(2)}}\right), \quad i \in [k], \label{sob121}
\end{equation}
which are also independent of each other. Furthermore, $Y_1$ and the $k$ strings of beads in \eqref{sob121} are also independent of the $(\alpha,0)$-string of beads 
\begin{equation*}
\left(\mu((\rho_{I_1}, X_1])^{-\alpha}(\rho_{I_1}, X_1], \mu((\rho_{I_1}, X_1])^{-1}\mu \restriction_{(\rho_{I_1}, X_1]}\right).
\end{equation*}
$X_2$ is now a random atom picked from $E^{(2)}=E^{(1)}\setminus (\rho_{I_1},X_1]$. By Lemma \ref{kone},
\begin{equation*} \left(\frac{\mu((\rho_{1}, \rho_{1}'))}{\mu(E^{(2)})}, \ldots,\frac{\mu((\rho_{I_1-1},\rho_{I_1-1}'))}{\mu(E^{(2)})},\frac{\mu((X_{1},\rho_{I_1}'))}{\mu(E^{(2)})}, \frac{\mu((\rho_{I_1+1}, \rho_{I_1+1}'))}{\mu(E^{(2)})}, \ldots,\frac{\mu((\rho_k, \rho_k'))}{\mu(E^{(2)})} \right)  \label{ms6}
\end{equation*} has a $\text{{\rm Dirichlet}}\left(\theta_{1}, \ldots,\theta_{I_1-1},\theta_{I_1}, \theta_{I_1+1},\ldots,\theta_{k}\right)$ distribution, and hence $X_2$ splits the set $E^{(2)}$ according to {\rm Dirichlet}$(\alpha,1-\alpha,\theta_{1},\ldots,\theta_{k})$. Therefore, inductively, for $n \geq 1$, we obtain the representation
\begin{equation*}
V_n:=\mu\left(\bigcup \limits_{m=1}^{n}(a_m, b_{m}] \right)=1-\prod \limits_{m=1}^n\left(1-Y_m\right),
\end{equation*}
where $(Y_m)_{m \geq 1}$ is a sequence of i.i.d. random variables with $Y_m \sim \text{\rm Beta}(1,\theta)$, using notation from \eqref{enotation}. Consider now the sets of interval components defined via $\left\{(a_n, b_{n}], n \geq 1\right\}$. Applying Lemma \ref{kone} inductively yields that the sequence $(Y_m)_{m \geq 1}$ is i.i.d. and independent of the $(\alpha,0)$-strings of beads
\begin{equation*}
\left(\mu((a_n, b_{n}])^{-\alpha}(a_n,b_{n}], \mu((a_n, b_{n}])^{-1}\mu \restriction_{(a_n,b_{n}]} \right), \quad n \geq 1.
\end{equation*}

We are now in the situation of Corollary \ref{intpartstick2}, and conclude that $(E_{\biguplus}=(\rho, \rho'), \mu)$ is indeed associated with an $(\alpha, \theta)$-regenerative interval partition of $[0,1]$. By Proposition \ref{sobconstruc}, $(E_{\biguplus}=(\rho, \rho'), \mu)$ equipped with $d_{\biguplus}$ defines an $(\alpha,\theta)$-string of beads.
\end{proof}

\section{Branch merging on (continuum) trees} \label{bm}

We introduce a branch merging operation on Ford CRTs, i.e. on binary fragmentation continuum random trees of the form $\mathcal T^{\alpha, 1-\alpha}$ arising in the scaling limit of the $(\alpha, \theta)$-tree growth process for $\theta=1-\alpha$. We construct a sequence of reduced trees $(\mathcal R_k^{\alpha,2-\alpha}, k \geq 1)$ associated with $\mathcal T^{\alpha, 2-\alpha}$, the scaling limit of the $(\alpha, 2- \alpha)$-tree growth process, based on branch merging and the $(\alpha,1-\alpha)$-coin tossing construction of embedded leaves according to the $(\alpha,1-\alpha)$-tree growth model. 

The $(\alpha,1-\alpha)$-coin tossing construction reduces to uniform sampling from the mass measure in the case of the Brownian CRT, i.e. when $\alpha=\theta=1/2$, which yields a simplified branch merging operation, which we use to construct a new tree-valued Markov process on a space of continuum trees.

\subsection{Preliminaries}
\subsubsection{$\mathbb{R}$-trees, self-similar continuum random trees and the Brownian CRT}
We recap some background on $\mathbb R$-trees. As we will only work with compact $\mathbb R$-trees, we include compactness in the definition. 
An \textit{$\mathbb R$-tree} is a compact metric space $(\mathcal T, d)$ such that the following two properties hold for every $\sigma_1, \sigma_2 \in \mathcal T$.
\begin{enumerate}
\item[(i)] There is an isometric map $\Phi_{\sigma_1, \sigma_2}: [0, d(\sigma_1, \sigma_2)] \rightarrow \mathcal T$ such that \begin{equation*} \Phi_{\sigma_1, \sigma_2}(0)=\sigma_1 \text{ and } \Phi_{\sigma_1, \sigma_2}(d(\sigma_1, \sigma_2))=\sigma_2. \end{equation*}
\item[(ii)] For every injective path $q: [0,1] \rightarrow \mathcal T$ with $q(0)=\sigma_1$ and $q(1)=\sigma_2$, \begin{equation*} q([0,1])=\Phi_{\sigma_1, \sigma_2}([0,d(\sigma_1, \sigma_2)]). \end{equation*}
\end{enumerate}
We write $[[\sigma_1, \sigma_2]]:=\Phi_{\sigma_1, \sigma_2}\left([0,d(\sigma_1, \sigma_2\right)])$ for the range of $\Phi_{\sigma_1, \sigma_2}$. 

A \textit{rooted} $\mathbb R$-tree $(\mathcal T, d, \rho)$ is an $\mathbb{R}$-tree $(\mathcal T, d)$ with a distinguished element $\rho$, the \textit{root}. In what follows, we only work with rooted $\mathbb{R}$-trees. Sometimes we refer to $\mathcal T$ as an $\mathbb R$-tree, the distance $d$ and the root $\rho$ being implicit. For any $\alpha > 0$ and any metric space $(\mathcal T, d)$, in particular any $\mathbb R$-tree, we write $\alpha \mathcal T$ for $(\mathcal T, \alpha d)$.

We only consider equivalence classes of rooted $\mathbb R$-trees. Two rooted $\mathbb R$-trees $(\mathcal T, d, \rho)$ and $(\mathcal T', d', \rho')$ are \textit{equivalent} if there exists an isometry from $\mathcal T$ onto $\mathcal T'$ such that $\rho$ is mapped onto $\rho'$. We denote by $\mathbb T$ the set of all equivalence classes of rooted $\mathbb R$-trees.
As shown in \citep{6}, the space $\mathbb T$ is a Polish space when endowed with the {pointed Gromov-Hausdorff distance} $d_{\rm GH}$. The \textit{pointed Gromov-Hausdorff distance} between two rooted $\mathbb{R}$-trees $(\mathcal T, d, \rho)$, $(\mathcal T',d', \rho')$ is defined as
\begin{equation*} d_{\rm GH} ((\mathcal T, d,\rho), (\mathcal T', d',\rho')) := \inf \{\max\{\delta(\phi(\rho), \phi'(\rho')), \delta_{H}(\phi(\mathcal T), \phi'(\mathcal T'))\}\}, \end{equation*}
where the infimum is taken over all metric spaces $(\mathcal M, \delta)$ and all isometric embeddings $\phi\colon \mathcal T \rightarrow \mathcal M$, $\phi' \colon \mathcal T' \rightarrow \mathcal M$ into $(\mathcal M, \delta)$, and $\delta_H$ is the Hausdorff distance between compact subsets of $(\mathcal M, \delta)$. In fact, the Gromov-Hausdorff distance only depends on equivalence classes of rooted $\mathbb R$-trees and so induces a metric on $\mathbb T$. We equip $\mathbb T$ with its Borel $\sigma$-algebra, and refer to \citep{6} for a formal construction of $\mathbb T$.

A \textit{weighted $\mathbb R$-tree} $(\mathcal T, d, \rho, \mu)$ is a rooted $\mathbb R$-tree $(\mathcal T, d, \rho)$ equipped with a probability measure $\mu$ on the Borel sets $\mathcal B(\mathcal T)$. Two {weighted $\mathbb R$-trees} $(\mathcal T, d, \rho, \mu)$ and $(\mathcal T', d', \rho', \mu')$ are called \textit{equivalent} if there exists an isometry from $(\mathcal T, d, \rho)$ onto $(\mathcal T', d', \rho')$ such that $\mu'$ is the push-forward of $\mu$. The set of equivalence classes of weighted $\mathbb R$-trees is denoted by $\mathbb{T}_{w}$. The Gromov-Hausdorff distance can be extended to a metric on $\mathbb T_w$, the Gromov-Hausdorff-Prokhorov metric, see e.g. \citep{38}.
For any rooted $\mathbb{R}$-tree $(\mathcal T,d,\rho)$ and $x \in \mathcal T$, we call $d(\rho, x)$ the \textit{height} of $x$, and $\sup_{x \in \mathcal T}d(\rho,x)$ the \textit{height} of $\mathcal T$.  A \textit{leaf} is an element $x \in \mathcal T$  with $x \neq \rho$ whose removal does not disconnect $\mathcal T$. We denote the set of all leaves of $\mathcal T$ by $\mathcal L(\mathcal T)$. An element $x \in \mathcal T$, $x \neq \rho$, is a \textit{branch point} if its removal disconnects the $\mathbb R$-tree into three or more components. $\rho$ is a branch point if its removal disconnects the $\mathbb R$-tree into two or more components. The \textit{degree} of a vertex $x \in \mathcal T$ is the number of connected components of $\mathcal T \setminus \{x\}$; we call an $\mathbb R$-tree \textit{binary} if the maximum branch point degree is $3$.

A pair $(\mathcal T, \mu)$ is a \textit{continuum tree} if $\mathcal T$ is an $\mathbb R$- tree, and $\mu$ is a probability measure on $\mathcal T$ satisfying the following three properties.
\begin{enumerate}
\item[(i)] $\mu$ is supported by $\mathcal L (\mathcal T)$, the set of leaves of $\mathcal T$.
\item[(ii)] $\mu$ has no atom, i.e. for any singleton $x \in \mathcal L (\mathcal T)$ we have $\mu(\{x\})=0$. 
\item[(iii)] For every $x \in \mathcal T \setminus \mathcal L(\mathcal T)$, $\mu(\mathcal T_x)>0$, where
\begin{equation*} \mathcal T_x := \{ \sigma \in \mathcal T: x \in [[\rho, \sigma]]\}. \end{equation*}
\end{enumerate}
By definition of a continuum tree, $\mathcal L (\mathcal T)$ is uncountable and has no isolated points. 

It will be useful to consider \textit{reduced} trees: for any rooted $\mathbb R$-tree $\mathcal T$ and any $x_1,x_2,\ldots,x_n \in \mathcal L(\mathcal T)$ let
\begin{equation} \mathcal R(\mathcal T, x_1,\ldots,x_n):= \bigcup \limits_{i=1}^n [[\rho, x_i]] \end{equation}
be the \textit{reduced} subtree associated with $\mathcal T, x_1,\ldots,x_n$. $\mathcal R(\mathcal T, x_1,\ldots,x_n)$ is an $\mathbb R$-tree, whose root is $\rho$ and whose set of leaves is $\{x_1,\ldots,x_n\}$.

For a weighted $\mathbb R$-tree $(\mathcal T, d, \rho, \mu)$ and $\mathcal R \subset \mathcal T$ closed, we define the projection \begin{equation} \pi^{\mathcal R}: \mathcal T \rightarrow \mathcal R, \qquad \sigma \mapsto \Phi_{\rho, \sigma}\left( \sup\{t \geq 0: \Phi_{\rho, \sigma}(t) \in \mathcal R\} \right), \end{equation}
where the function $\Phi_{\rho, \sigma}: [0, d(\rho, \sigma)] \rightarrow \mathcal T$ is the isometry with $\Phi_{\rho, \sigma}(0)=\rho$ and $\Phi_{\rho, \sigma}\left(d(\rho, \sigma)\right)=\sigma$. We write
\begin{equation*} \mu_{\mathcal R}(C):=\mu \left(\left(\pi^{\mathcal R}\right)^{-1}(C)\right), \quad C \in \mathcal B\left(\mathcal R\right), \end{equation*} 
for the push-forward of $\mu$ via $\pi^{\mathcal R}$, where $\mathcal B\left(\mathcal R\right)$ is the Borel $\sigma$-algebra on $\mathcal R$.

Weighted $\mathbb R$-trees can be obtained from \textit{height} \textit{functions}, that is, continuous functions $h: [0, 1] \rightarrow [0, \infty)$ with $h(0)=h(1)=0$. We use $h$ to define a distance by
\begin{equation} d_{h}(x,y):=h(x)+h(y)-2 \inf \limits_{\min\{x,y\} \leq z \leq \max \{x,y\}} h(z), \quad x,y \in  [0,1]. \label{height} \end{equation} Let $y \sim y'$ if $d_h(y,y')=0$, and take the quotient $\mathcal T_h=[0, 1)/\sim$. Then $(\mathcal T_h, d_h, \rho)$ is a rooted $\mathbb R$-tree coded by the function $h$, where the root $\rho$ is the equivalence class of $0$.
When the height function $h:[0,1]\rightarrow [0, \infty)$ is random we obtain a \textit{random} $\mathbb{R}$-tree  $(\mathcal T_h, d_h, \rho)$. Note that (random) $\mathbb{R}$-trees can be equipped with a natural mass measure $\mu_{h}$ induced by the Lebesgue measure on $[0,1]$. 

We study a certain type of random $\mathbb R$-trees, namely binary fragmentation continuum random trees, arising in the scaling limit of $(\alpha, \theta)$-tree growth processes.

\begin{definition}[Binary fragmentation continuum random tree]\rm
A random weighted rooted binary $\mathbb{R}$--tree $(\mathcal T, d, \rho, \mu)$ is called a \textit{binary fragmentation continuum random tree (CRT)} of index $\gamma > 0$, if 
\begin{enumerate} \item[(i)] the measure $\mu$ has no atoms and $\mu(\mathcal T_x) >0$ a.s. for all $x \in \mathcal T \setminus \mathcal L(\mathcal T)$, where
\begin{equation*} \mathcal T_{x} = \{\sigma \in \mathcal T: x \in [[\rho, \sigma]]\}, \end{equation*}  and $\mu([[\rho, x]])=0$ for all $x \in \mathcal T$;
\item[(ii)] for each $t \geq 0$, given the masses of the connected components $(\mathcal T_i^t, i \geq 1)$ of $\{\sigma \in \mathcal T: d(\rho,\sigma) > t\}$, indexed in decreasing order of mass and completed by a root vertex $\rho_i$, i.e. given $(\mu(\mathcal T_i^t), i \geq 1) = (m_i, i \geq 1)$ for some $m_1 \geq m_2 \geq \cdots \geq 0$, the rescaled trees 
\begin{equation*} \left(\mathcal T_i^t, m_i^{-\gamma} d \restriction_{\mathcal T_i^t}, \rho_i, m_i^{-1} \mu \restriction_{\mathcal T_i^t} \right) \end{equation*} are independent identically distributed isometric copies of $(\mathcal T, d, \rho , \mu)$. \end{enumerate}
\end{definition}

Binary fragmentation CRTs are characterized by a self-similarity parameter \linebreak $\alpha >0$ and a so-called dislocation measure $\nu(du)$, that is, a $\sigma$-finite measure on $[1/2,1)$ satisfying $\int_{[1/2,1)}(1-u)\nu(du) < \infty$. They can be decomposed into isometric i.i.d. copies of the (in a probabilistic sense) ``same'' tree, when split into subtrees along the spine from the root to a leaf sampled from the mass measure on the CRT.

\begin{theorem}[Spinal decomposition, \citep{1}, Proposition 18] \label{spinaldec}
Let $(\mathcal T, d, \rho, \mu)$ be a binary fragmentation CRT with self-similarity parameter $\gamma >0$.
Select a leaf $\Sigma^* \sim \mu$ in $\mathcal T$, consider the spine $[[\rho, \Sigma^*]]$ and the  connected components $(\mathcal T_i, i \in I)$ of $\mathcal T \setminus [[\rho, \Sigma^*]]$, each completed by a root vertex $\rho_i$. Furthermore, assign mass $m_i:=\mu(\mathcal T_i)$ to each point $\rho_i \in [[\rho, \Sigma^*]]$ and denote the resulting distribution on $[[\rho, \Sigma^*]]$ by $\mu^*$. Then, given the string of beads $([[\rho, \Sigma^*]], \mu^*)$, the rescaled trees
\begin{equation*} \left(\mathcal T_i, m_i^{-\gamma}d \restriction_{\mathcal T_i}, \rho_i, m_i^{-1}\mu \restriction_{\mathcal T_i}\right), \qquad i \in I, \end{equation*}
are independent identically distributed isometric copies of $(\mathcal T, d, \rho,\mu)$. 
\end{theorem}

For any $\Sigma \in \mathcal L(\mathcal T)$, we say that the \textit{spinal decomposition theorem holds for the spine $[[\rho, \Sigma]]$} to mean that Theorem \ref{spinaldec} holds when we replace the leaf $\Sigma^*$ by $\Sigma$.

Examples for binary fragmentation CRTs include Ford CRTs, i.e. the trees of the form $\mathcal T^{\alpha,1-\alpha}$, $\alpha \in (0,1)$, with self-similarity index $\alpha=\gamma$, see Section \ref{alphathetamodel}.
When $\gamma=\alpha=1/2$, the tree $\mathcal T^{1/2,1/2}$ is the Brownian CRT which can be defined in terms of Brownian excursion, see e.g. \citep{18}. Specifically, let $W=(W(t), 0 \leq t \leq 1)$ be a standard Brownian excursion. The tree $(\mathcal T_{2W}, d_{2W}, \rho)$ defined via $2W$ as height function, with mass measure $\mu_{2W}$ induced by the Lebesgue measure on $[0,1]$, is called the \textit{Brownian continuum random tree} (Brownian CRT), see \cite{9}. 

We often need a random $\mathbb R$-tree whose equivalence class has the same distribution as a Ford CRT on $\mathbb T$ (or $\mathbb T_w$). We also refer to such $\mathbb R$-trees as Ford CRTs.

\subsubsection{Regenerative tree growth: the $(\alpha, \theta)$-model} \label{alphathetamodel}

Ordered $(\alpha, \theta)$-CRPs and related $(\alpha, \theta)$-strings of beads naturally appear in the study of the $(\alpha, \theta)$-tree growth process $(T_n^{\alpha, \theta}, n \geq 1)$, as studied by \cite{1}.  

Consider the set $\mathbb T_n^b$ of rooted binary trees with $n$ leaves labelled by $[n]=\{1,\ldots,n\}$. We study the growth model given by the following growth procedure. Let $T_1$ be the tree consisting of one single edge, joining the root $\rho$ and the leaf with label $1$. At step $2$ we select this edge and split it into a Y-shaped tree $T_2$, which has one edge that connects the root vertex and a binary branch point; there are two leaves labelled by $1$ and $2$ which are linked to the single branch point by one edge each. Now, inductively, to construct $T_{n+1}$ conditionally given $T_n$, select an edge of $T_n$ according to some weights (which we specify later). We select the edge $\Omega_n \rightarrow \Omega''_n$, say, directed away from the root. We replace this edge by \textit{three} edges $\Omega_n \rightarrow \Omega_n'$, $\Omega'_n \rightarrow n+1$ and $\Omega'_n \rightarrow \Omega''_n$, meaning that there is one edge connecting $\Omega_n$ to a new branch point $\Omega'_n$ which is linked to the existing vertex $\Omega''_n$ and a new leaf $n+1$. Clearly, $T_n \in \mathbb T_n^b$ for all $n \geq 2$. 

The stochastic process $\left(T_n, n \geq 1\right)$ of random trees created in the described way, where each edge is selected randomly according to some selection rule, is called a \textit{binary tree growth process}. A range of selection rules, which are partly related to each other, was studied in the literature. We focus on the $(\alpha, \theta)$-selection rule.

\begin{definition}[$(\alpha, \theta)$-tree growth process] \label{deftgp}\rm
For $0 \leq \alpha < 1$ and $\theta \geq 0$, the $(\alpha, \theta)$-selection rule is defined as follows.
\begin{enumerate} \item[(i)] For any $n \geq 2$, consider the branch point of the tree $T_n$ adjacent of the root, and the two subtrees $T_{n,0}$ and $T_{n,1}$ above this point, where $T_{n,1}$ contains the smallest label in $T_n$ . Denote their sizes, i.e. the number of leaves, by $m$ and $n-m$, respectively. Assign weight $\alpha$ to the edge connecting the root and the adjacent branch point, and weights $m-\alpha$, $n-m-1+\theta$ to the subtrees $T_{n,0}$, $T_{n,1}$, respectively.
\item[(ii)] Choose one of the two subtrees or the edge adjacent to the root proportionally to these weights (note that the total weight is $n-1+\theta$). If a subtree with two or more leaves was selected, recursively apply the weighting procedure and the random selection until an edge or a subtree with a single leaf is chosen. If a subtree with a single leaf was selected, select the unique edge of this subtree. 
\end{enumerate}

A binary tree growth process $(T_n^{\alpha, \theta}, n \geq 1)$ grown via the $(\alpha, \theta)$-rule for some $0\leq \alpha < 1$ and $\theta \geq 0$ is called an \textit{$(\alpha, \theta)$-tree growth process}, see Figure \ref{treeex} for an example showing the growth procedure.

We write $(T_n^{\alpha, \theta,o}, n \geq 1)$ for the delabelled tree growth process. The trees  $T_n^{\alpha, \theta}$ and $T_n^{\alpha, \theta,o}, n \geq 1$, are considered as $\mathbb{R}$--trees with unit edge lengths.
\end{definition}

\begin{remark}[Ford's rule] \rm
The selection rule for $\theta=1-\alpha$, $0 \leq \alpha < 1$, is called \textit{Ford's rule}, see \cite{12}.
\end{remark}

Consider the spine $[[\rho, 1]]$, i.e. the unique path in $T^{\alpha, \theta}_n$ connecting the root and leaf $1$.  Referring to the subtrees along the spine as \textit{tables} and the leaves within these subtrees as \textit{customers} we obtain an ordered $(\alpha, \theta)$-CRP with label set $\{2,3, \ldots\}$. Whenever an edge on the spine is selected, the height of leaf $1$, that is, the graph distance between the root $\rho$ and leaf $1$, i.e. the number of tables $K_n$ in the CRP, increases by $1$. Note that the $j$-th customer in the restaurant carries label $(j + 1)$ as a leaf in the tree, since leaf $1$ is not contained in a subtree off the spine. Each subtree is uniquely characterized by its leaf labels which is a natural consequence of the growth procedure of $T_n^{\alpha, \theta}$, $n \geq 1$. Moreover, the order of tables on the spine (or \textit{spinal} order) is consistent across different values for $n$.

$(\alpha, \theta)$-tree growth processes are \textit{regenerative} in the following sense.

\begin{definition}[Regenerative tree growth process]\rm
A {binary} tree growth process $(T_n, n \geq 1)$ is called \textit{regenerative} if for each $n \geq 2$, conditionally given that the first split of $T_n$ is into two subtrees $T_{n,0}$, $T_{n,1}$ with label sets $B_0, B_1$, respectively, the relabelled trees  $\widetilde{T}_{n,0}$ and $\widetilde{T}_{n,1}$ are independent copies of $T_{\#B_i}$, $i=0,1$, respectively, where $\#B_i$ denotes the cardinality of the set $B_i$, $i=0,1$, and $\widetilde{T}_{n,i}$ is the tree $T_{n,i}$ with leaves relabelled by the increasing bijection $B_i \rightarrow [\#B_i]$, $i=0,1$.
\end{definition}

The trees grown via the $(\alpha, \theta)$-selection rule have leaves which can be identified by successively assigned labels from $\{1,2,\ldots\}$. \cite{1} studied exchangeability of leaf labelling.

\begin{lemma}[\citep{1}, Proposition 1]\label{exchange} Let $(T_n^{\alpha, \theta}, n \geq 1)$  be an $(\alpha, \theta)$-tree growth process for some $\alpha \in (0,1) $ and $\theta \geq 0$. The leaf labels of $T_n^{\alpha, \theta}$
are exchangeable for all $n \geq 1$ if and only if $\alpha=\theta= 1/2$.
\end{lemma}

\begin{figure}[htb] \centering \def\svgwidth{\columnwidth} 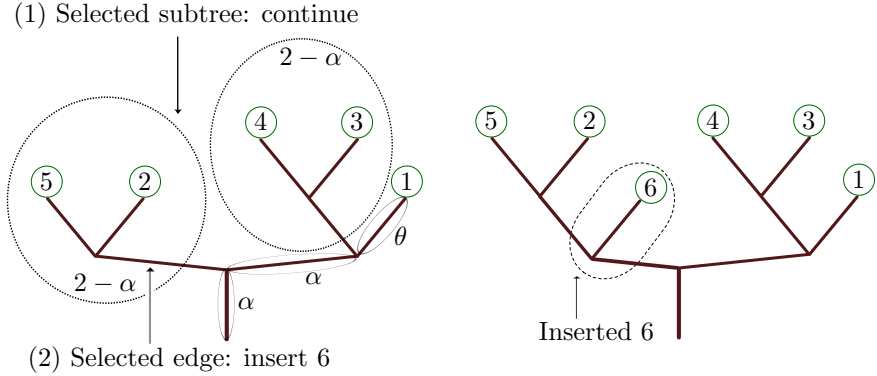 \caption{ Example of two-step recursion in the $(\alpha, \theta)$-model. First subtree is selected with probability $(2-\alpha)/(4+\theta)$; within this subtree leaf $6$ is inserted at the root edge with probability $\alpha/(1+\theta)$.}  \label{treeex} \end{figure}

Limits for the delabelled tree growth process $(T_n^{\alpha, \theta,o}, n\geq 1)$ were established in \citep{1} by considering the reduced trees $\mathcal R(T_n^{\alpha, \theta}, [k])$, $k \geq 1$, where each edge between two vertices $v_1$ and $v_2$ is equipped with the graph distance in $T_n^{\alpha, \theta}$, i.e. with the number of edges between $v_1$ and $v_2$ in $T_n^{\alpha, \theta}$. 

\begin{proposition}[\citep{1}, Proposition 2]\label{propred} Let $(T_n^{\alpha, \theta}, n\geq 1)$ be an $(\alpha, \theta)$-tree growth process for some $0 < \alpha <1$ and $\theta \geq 0$. Then, for all $k\geq1$,
\begin{equation} \lim \limits_{n \rightarrow \infty} n^{-\alpha} \mathcal R(T_n^{\alpha, \theta}, [k])= \mathcal R_k^{\alpha, \theta} \qquad \text{a.s.}, \end{equation}
i.e. the edge lengths of $\mathcal R(T_n^{\alpha, \theta}, [k])$ scaled by $n^{-\alpha}$ converge jointly a.s. as $n \rightarrow \infty$ to the edge lengths of a tree $\mathcal R_k^{\alpha, \theta}$ with the same shape as $\mathcal R(T_n^{\alpha, \theta}, [k])$.
\end{proposition}

Recall the notation $\mathcal R_k^{\alpha, \theta,o}$ for the delabelled tree associated with $\mathcal R_k^{\alpha, \theta}$. 

\begin{theorem}[\citep{1}, Theorem 3]\label{atheta}
Consider the situation from Proposition \ref{propred}, and the delabelled trees $\mathcal R_k^{\alpha, \theta,o}$ associated with $\mathcal R_k^{\alpha, \theta}$ for $k \geq 1$.  Then there exists a CRT $\mathcal T^{\alpha, \theta}$ on the same probability space such that
\begin{equation}
\lim \limits_{k \rightarrow \infty} \mathcal R_k^{\alpha, \theta,o} = \mathcal T^{\alpha, \theta} \qquad \text{a.s. in the Gromov-Hausdorff topology}.
\end{equation}
\end{theorem}

In general, the trees $\mathcal T^{\alpha, \theta}$ are binary fragmentation CRTs, characterized by the self-similarity parameter $\alpha$, and the absolutely continuous dislocation measure $\nu_{\alpha, \theta}(du)=f_{\alpha, \theta}^{o}(u)du$ for $u \in (1/2, 1)$, where the \textit{dislocation density} is given via 
\begin{align} \Gamma(1-\alpha) f_{\alpha, \theta}^{o}(u)&= \alpha \left( u^{\theta} (1-u)^{-\alpha-1}+u^{-\alpha-1} (1-u)^{\theta} \right) \nonumber \\ & \quad+ \theta \left( u^{\theta-1} (1-u)^{-\alpha}+u^{-\alpha} (1-u)^{\theta-1} \right), \quad 1/2 < u < 1.
\label{disdens}
\end{align}

\begin{corollary}\label{fords}
For any $\alpha \in (0,1)$, $\mathcal T^{\alpha, 1-\alpha} \,{\buildrel d \over =}\, \mathcal T^{\alpha,2-\alpha}$. \end{corollary}
The family of trees $\{\mathcal T^{\alpha,1-\alpha}, \alpha \in (0,1)\}$ is called \textit{Ford CRTs}.
\begin{proof}
This was shown in \citep{1}.
\begin{itemize}
\item The marginal distributions of delabelled trees coincide for the regenerative $(\alpha, 1-\alpha)$- and $(\alpha, 2-\alpha)$-tree growth processes (see Lemma 12 in \citep{1}).
\item The $(\alpha, 1-\alpha)$- and the $(\alpha, 2-\alpha)$-tree growth processes have the trees $\mathcal T^{\alpha,1-\alpha}$ and $\mathcal T^{\alpha,2-\alpha}$ as their distributional scaling limits (Proposition \ref{propred}).\qedhere \end{itemize} \end{proof}

\begin{corollary} \label{BCRTScal}
For $\alpha=1/2$ and $\theta \in \{1/2,3/2\}$, $\mathcal T^{\alpha, \theta}$ is the Brownian CRT. \end{corollary}
\begin{proof} It was shown in \citep{2} that the delabelled $(1/2, 1/2)$-tree growth process has the Brownian CRT as its distributional scaling limit.  This also follows from earlier work by \cite{7}, due to the connection between the $(1/2, 1/2)$-tree growth process and conditioned critical binary Galton-Watson trees, whose scaling limit is the Brownian CRT.

By Corollary \ref{fords}, the Brownian CRT is also the limit of the delabelled $(1/2, 3/2)$-tree growth process. T
\end{proof}

\cite{1} used an embedding of the reduced trees $\mathcal R_k^{\alpha, \theta}$, $k \geq 1$, into the CRT $\mathcal T^{\alpha, \theta}$ to prove convergence results for $(T^{\alpha, \theta}_n, n \geq 1)$.

\begin{theorem} \label{embredtrees2}
Let $(\mathcal T^{\alpha, \theta}, d, \rho, \mu)$ be a CRT as in Theorem \ref{atheta}. Then there exists a sequence of leaves $\Sigma_k, k \geq 1$, on a suitably enlarged probability space such that 
\begin{equation} (\mathcal R(\mathcal T^{\alpha, \theta}, \Sigma_1, \ldots, \Sigma_k), k \geq 1) \,{\buildrel d \over =}\, (\mathcal R_k^{\alpha, \theta}, k \geq 1). \end{equation} 
\end{theorem}

The problem of finding the leaves $\Sigma_k, k \geq 1$, as in Theorem \ref{embredtrees2} inside a CRT was solved in \citep{1} (Lemma 19, Proposition 20 and Corollary 21). We refer to this problem as the \textit{leaf embedding problem.} In particular, given a CRT $\mathcal T=\mathcal T^{\alpha,\theta}$ for some $\alpha \in (0,1), \theta >0,$ with mass measure $\mu$, it involves the following steps.
\begin{itemize}
\item[(i)] Find a leaf $\Sigma_1 \in \mathcal T$ such that $([[\rho, \Sigma_1]], \mu_{[[\rho, \Sigma_1]]})$ is an $(\alpha, \theta)$-string of beads, i.e. embed $(\mathcal R_1, \mu_1)$ into $(\mathcal T, \mu)$ via  $(\mathcal R_1, \mu_1)=([[\rho, \Sigma_1]], \mu_{[[\rho, \Sigma_1]]})$.
\item[(ii)] Given $(\mathcal R_1, \mu_1)$, find the point on $\mathcal R_1$ which identifies the root $J_1$ of the component of $\mathcal T \setminus \mathcal R_1$ which contains $\Sigma_2$, so that $\mathcal R_2 = \mathcal R_1 \cup ]]J_1, \Sigma_2]]$.
\item [(iii)] Iterate the procedure to embed $(\mathcal R_k, \mu_k)$ for general $k \geq 2$ into $\mathcal T$.
\end{itemize}

When $(\alpha, \theta)=(1/2,1/2)$ the leaf embedding problem can be solved by sampling leaves $\Sigma_k, k \geq 1$, independently from the mass measure $\mu$. The general technique to solve the leaf embedding problem is more involved, but not needed to understand the Branch Merging Markov Chain, Section \ref{bmc}, and the new and simple leaf embedding approach for $(1/2, 3/2)$-tree growth processes in Section \ref{recbm}. 

We recap the procedure to solve (i) for $\alpha \in (0,1)$, $\theta >0$, from \citep{1}, Section 4.3. 
\begin{itemize}
\item Let $(\mathcal T^{(1)}, d^{(1)}, \rho^{(1)}, \mu^{(1)})=(\mathcal T, d, \rho, \mu)$ and let $\Sigma_1^{*(1)}$ be a random leaf sampled from $\mu^{(1)}$ in $\mathcal T^{(1)}$. Consider the 
string of beads \begin{equation*} \left([[\rho, \Sigma_1^{*(1)}]], \mu_{[[\rho, \Sigma_1^{*(1)}]]}\right), \end{equation*} which can be represented via a local time process $(L^{*(1)}(u), 0 \leq u \leq 1)$, i.e. $d(\rho, \Sigma_1^{*(1)})=L^{*(1)}(1)$ and
\begin{align*}
\mu_{[[\rho, \Sigma_1^{*(1)}]]}\left(\left\{ \Phi_{\rho, \Sigma_1^{*(1)}}\left(L^{*(1)}\left(1-e^{-\xi_t^{*(1)}}\right)\right)\right\}\right)=e^{-\xi_{t^-}^{*(1)}}-e^{-\xi_{t}^{*(1)}}=-\Delta e^{-\xi_t^{*(1)}}, \end{align*}
for a subordinator $(\xi_t^{*(1)}, t \geq 0)$, the so-called \textit{spinal subordinator}, whose distribution was characterized in \citep{1}, Section 4.1.

Consider the switching time ${\tau}^{(1)}$ associated with the switching probability function $(\hat{p}(u), 0 \leq u \leq 1)$ defined by
\begin{equation*} \hat{p}(u):=\frac{(1-u)f_{\alpha, \theta}(1-u)}{f_{\alpha, \theta}^*(u)}, \quad 0 < u < 1,
\end{equation*}
as in Proposition \ref{cointoss} where, for $0 < u < 1$, $f_{\alpha, \theta}(u)$ and $f_{\alpha, \theta}^*(u)$ are given by
\begin{align*} f_{\alpha, \theta}&=\Gamma(1-\alpha)^{-1}\left( \alpha(1-u)^{-\alpha-1}u^{\theta-1} + \theta u^{\theta-2} (1-u)^{-\alpha} \right), \\
f^*_{\alpha, \theta}(u)&=uf_{\alpha,\theta}(u)+(1-u)f_{\alpha,\theta}(1-u). \end{align*}
Note that $f^o_{\alpha, \theta}(u)=uf_{\alpha, \theta}(u) + (1-u)f_{\alpha,\theta}(1-u), 1/2 < u <1$, where $f_{\alpha, \theta}^o$ was given in \eqref{disdens}. We refer to \citep{62} (19) for a derivation of $f_{\alpha, \theta}^*$. Define \begin{equation*} F_t:=\exp(-\Delta\xi_t^{*(1)}), \quad 0 \leq t < \tau^{(1)}, \qquad F_{\tau^{(1)}}=1-\exp(-\Delta\xi_{\tau^{(1)}}^{*(1)}). \end{equation*}
\item For $i \geq 1$, consider the interval partition $\{1- \exp(-\xi_t^{*(i)}), t \geq 0 \}^{\rm cl}$, with associated local time process $(L^{*(i)}(u), 0 \leq u \leq 1)$, and the junction point
\begin{equation*} \rho^{(i+1)}=\Phi_{\rho^{(i)},\Sigma_1^{*(i)}}\left(1-L^{*(i)}\left(\exp(-\xi_{\tau^{(i)}}^{*(i)})\right)\right). \end{equation*} Furthermore, define
\begin{align*}
\mathcal T^{(i+1)}&=\left\{ \sigma \in \mathcal T^{(i)}: [[\rho^{(i)}, \sigma]] \cap [[\rho^{(i)}, \Sigma_1^{*(i)}]]=[[\rho^{(i)}, \rho^{(i+1)}]] \right\}, \\
d^{(i+1)} &=\left( 1-  \exp \left(-\xi_{\tau^{(i)}-\tau^{(i-1)}}^{*(i)} \right) \right) ^{-\alpha}d^{(i)} \restriction_{\mathcal T^{(i+1)}}, \\
\mu^{(i+1)}&=\left(1- \exp\left(-\xi_{\tau^{(i)}-\tau^{(i-1)}}^{*(i)} \right)\right)^{-1} \mu^{(i)} \restriction_{\mathcal T^{(i+1)}}.
\end{align*}
Sample a leaf $\Sigma_1^{*(i+1)}$ from $\mu^{(i+1)}$ in $\mathcal T^{(i+1)}$, and proceed as before to determine the spinal subordinator $\xi^{*(i+1)}$ and the switching time $\tau^{(i+1)}$. Set
\begin{align*} F_{\tau^{(i)}+t}&=\exp(-\Delta\xi_t^{*(i+1)}), \quad 0 \leq t < \tau^{(i+1)}- \tau^{(i)}, \\ F_{\tau^{(i+1)}}&=1-\exp(-\Delta\xi_{\tau^{(i+1)}- \tau^{(i)}}^{*(i+1)}). \end{align*}
\end{itemize}

\begin{lemma}[\citep{1}, Proposition 20]\label{leafemblem} The process $(F_t, t \geq 0)$ is a Poisson point process in $(0,1)$ with intensity measure $uf_{\alpha, \theta}(u)$ and cemetery state $1$. Furthermore, consider the space \begin{equation*} [[\rho, \Sigma_1[[:=\bigcup_{i \geq 1}[[\rho, \rho^{(i)}]]. \end{equation*} Then $\Sigma_1 \in \mathcal T$ is a leaf a.s., the string of beads $([[\rho, \Sigma_1]], \mu_{[[\rho, \Sigma_1]]})$ is a weight-preserving isometric copy of $(\mathcal R_1, \mu_1)$, and the spinal decomposition theorem holds for the spine $[[\rho, \Sigma_1]]$ with connected components $(\mathcal T_i, i \in I)$ of $\mathcal T \setminus [[\rho, \Sigma_1]]$.
\end{lemma}

We refer to this procedure to find the leaf $\Sigma_1$ as the \textit{$(\alpha,\theta)$-coin tossing construction} (or simply the \textit{coin tossing construction}) in the tree $(\mathcal T, \mu)$.

In order to find the branch point where the spine to leaf $\Sigma_1$ leaves the spine to leaf $\Sigma_2$, we apply $(\alpha,\theta)$-coin tossing sampling in the $(\alpha, \theta)$-string of beads $(\mathcal R_1, \mu_1)$ which we have already embedded into $(\mathcal T, \mu)$. By the spinal decomposition theorem, we conclude that the subtrees associated with that atom as branch point base point is a rescaled independent copy of $(\mathcal T, \mu)$, and hence we can apply the  $(\alpha,\theta)$-coin tossing construction in this subtree to identifiy $\Sigma_2$. Therefore, given the embedding of $(\mathcal R_1, \mu_1)$ into $(\mathcal T, \mu)$, we can embedd $(\mathcal R_2, \mu_2)$ by identifying leaf $\Sigma_2$, and considering $\mu_2$ as the projection of the measure $\mu$ onto $\mathcal R_2$. It was shown in \citep{1}, Corollary 21, that this method can be used to embed $(\mathcal R_k, \mu_k)$ into $(\mathcal T, \mu)$ for all $k \geq 1$.

\begin{lemma}[Embedding of $(\mathcal R_k, \mu_k)$, \citep{1}, Corollary 21] \label{embedall} For $k \geq 1$, given an embedding of $(\mathcal R_k, \mu_k)$ into $(\mathcal T, \mu)$, perform the following steps to embed $(\mathcal R_{k+1}, \mu_{k+1})$.
\begin{itemize}
\item Select an edge $e$ of $\mathcal R_k$ proportionally to the mass assigned to $e$ by $\mu_k$.
\item If $e$ is an internal edge of $\mathcal R_k$, pick an atom $J_k$ from the mass measure $\mu_k$ restricted to $e$; otherwise perform $(\alpha, \theta)$-coin tossing sampling (see Proposition \ref{cointoss}) in the string of beads  $(\mu_k(e)^{-\alpha}e, \mu_k(e)^{-1}\mu_k \restriction_{e})$ to identify $J_k$.
\item Define the pair $(\mathcal R_{k+1}, \mu_{k+1})$ by
\begin{equation*} \mathcal R_{k+1}:=\mathcal R_k \cup [[J_k, \Sigma_{k+1}]], \quad \mu_{k+1}:=\mu_{\mathcal R_{k+1}}. \end{equation*}
\end{itemize}
\end{lemma}



\subsection{The Branch Merging Markov Chain}
\label{bmc}
We present a branch merging operation on continuum trees, and use it as transition rule for a Markov process operating on the space of continuum trees, the so-called Branch Merging Markov Chain. 

Given a continuum tree, we pick two leaves independently from the mass measure on the tree, merge the three branches of the arising reduced Y-shaped tree equipped with projected subtree masses according to Algorithm \ref{algostring}, and replant the subtrees on the created longer branch. This is in principle a simplified version of the more elaborate and general branch merging operation presented in Section \ref{genbm}, and, in contrast to Section \ref{genbm}, this branch merging operation is detached from leaf embedding.

\begin{definition}[Branch Merging Markov Chain] \label{defbmmc} The \textit{Branch Merging Markov Chain (BMMC)} is the time-homogeneous Markov Chain $(\mathcal B_k, k \geq 0)$ on the space of continuum trees which is defined via the following transition rule. Given a continuum tree $\mathcal B_0:=(\mathcal T, d, \rho, \mu)$, obtain a tree $\mathcal B_1:=(\widetilde{\mathcal T}, \tilde{d}, \rho, \tilde{\mu})$ as follows.
\begin{enumerate}
\item[(i)$^{\text{MC}}$] \textit{Start configuration}. Pick two leaves $\Sigma_1$ and $\widetilde{\Sigma}_1$ independently by uniform sampling from the mass measure $\mu$ on $\mathcal T$, and denote by ${\Omega}$ the branch point of the spines from the root $\rho$ to  $\Sigma_1$ and $\widetilde{\Sigma}_1$. We write $\mathcal R_2=\mathcal R(\mathcal T, \Sigma_1, \widetilde{\Sigma}_1)$, and obtain the three branches
\begin{equation} E_1:=[[\rho, {\Omega}[[,\quad E_2:=]]{\Omega}, \Sigma_1[[,\quad E_{3}:=[[{\Omega}, \widetilde{\Sigma}_1[[. \label{ibranches} \end{equation} Equip the tree $\mathcal R_2$ with the mass measure $\mu_{\mathcal R_2}$ capturing the masses of the connected components $\mathcal S_x, x \in \mathcal R_2,$ of $\mathcal T \setminus \mathcal R_2$ projected onto $\mathcal R_2$ where $\mathcal S_x$ is the subtree of $\mathcal T$ rooted at $x \in \mathcal R_2$, i.e. $\mu_{\mathcal R_2}(x) = \mu(\mathcal S_x)$.

\item[(ii)$^{\text{MC}}$] \textit{Cut point sampling}. \label{step2} Determine a sequence of cut points $(X_k)_{k \geq 1}$ for the strings of beads $(E_i, \mu_{\mathcal R_2} \restriction_{E_i})$, $i \in [3]$,  as follows. 
\begin{itemize} \item Set $E_i^{(1)}:=E_i$, $i \in [3]$.
\item For $k \geq 1$, conditionally given $E_i^{(k)}, i \in [3]$, sample $X_{k}$ from the normalised mass measure $\mu_{\mathcal R_2}(E^{(k)})^{-1} \mu_{\mathcal R_2}\restriction_{E^{(k)}}$ on $E^{(k)}=\bigcup_{i \in [3]} E_i^{(k)}$. Set $$E_i^{(k+1)} = \begin{cases}E_i^{(k)} \setminus [[\rho_i, X_k]]& \text{ if }X_k \in E_i^{(k)}, \\  E_i^{(k)}  &\text{ if }X_k \notin E_i^{(k)}, \end{cases}$$
for $i \in [3]$ where $\rho_i$ is the left endpoint of $E_i$ in \eqref{ibranches}, $i \in [3]$.
\end{itemize}

\item[(iii)$^{\text{MC}}$] \textit{Tree pruning and spine merging}. \label{step3} 
Merge the strings of beads $(E_i, \mu_{\mathcal R_2} \restriction_{E_i})$, $i \in [3]$, using the set of cut points $\Upsilon=\{X_k, k \geq 1\}$ (cf. Section \ref{mergalg}) to obtain a merged string of beads $(\widetilde{\mathcal R}_1, {\mu}_{\widetilde{\mathcal R}_1})$ connecting the root $\rho$ and the leaf $\Sigma_1$.

\item[(iv)$^{\text{MC}}$] \textit{Subtree replanting}. \label{step4}   Replant the trees $\mathcal S_x, x \in \mathcal R_2,$ at their root vertices $x \in \widetilde{\mathcal R}_1$ to obtain the output tree $\mathcal B_1=(\widetilde{\mathcal T}_1, \tilde{d}_1, \rho, \tilde{\mu}_1)$ where $\tilde{d}_1$ and $\tilde{\mu}_1$ denote the metric and the mass measure on $\widetilde{\mathcal T}_1$ pushed forward via this operation. \end{enumerate}
\end{definition}

\begin{remark}
Note that, in the cut point sampling procedure ${\rm (ii)}^{\rm MC}$, we subsequently cut off the bottom parts of the branches that are hit in the sampling procedure. This was presented in the context of $(\alpha, \theta)$-strings of beads in Algorithm \ref{algostring}(ii)-(iii). In contrast to Algorithm \ref{algostring}, we work with general strings of beads (not necessarily regenerative) and uniform sampling, corresponding to $(1/2,1/2)$-coin tossing sampling.
\end{remark}
The merging and replanting procedure $\rm{(iii)}^{\rm MC}$-$\rm{(iv)}^{\rm MC}$  can be defined explicitely as follows. Define the merged branch $\widetilde{\mathcal R}_1$ via \begin{equation*} \widetilde{\mathcal R}_1:=\{\rho \} \cup \left(\biguplus \limits_{i\in[3]} E_i\right) \cup \{\Sigma_1\}, \end{equation*}
where the operator $\biguplus$ was given in \eqref{operator}, Section \ref{mergalg}, and equip $\widetilde{\mathcal R}_1$ with the metric $d_{\widetilde{\mathcal R}_1}: \widetilde{\mathcal R}_1 \times \widetilde{\mathcal R}_1 \rightarrow \mathbb R_0^{+}$, $d_{\widetilde{\mathcal R}_1}(x,y):=\lvert d_{\widetilde{\mathcal R}_1}(\rho, x)-d_{\widetilde{\mathcal R}_1}(\rho, y)\rvert$ where
\begin{equation*}
d_{\widetilde{\mathcal R}_1}(\rho,x):=d_{\biguplus}(\rho, x) \text{ for } x \in \widetilde{\mathcal R}_1 \setminus \{ \Sigma_1\}, \quad d_{\widetilde{\mathcal R}_1}(\rho,\Sigma_1)=d(\rho, \Sigma_1) + d(\Omega, \widetilde{\Sigma}_1). \end{equation*}
We get the metric space $(\widetilde{\mathcal T}_1,\tilde{d}_1)$ where $\widetilde{\mathcal T}_1:=\widetilde{\mathcal R}_1 \cup \bigcup _{x \in \widetilde{\mathcal R}_1} \mathcal S_x$ and $\widetilde{d}_1: \widetilde{\mathcal T}_1 \times \widetilde{\mathcal T}_1  \rightarrow \mathbb R_{0}^+$,
\begin{equation}
\tilde{d}_1(x,y):= \begin{cases} d_{\widetilde{\mathcal R}_1}(x,y) & \text{ if } x, y \in \widetilde{\mathcal R}_1,\\
 d_{\widetilde{\mathcal R}_1}(x',y')+d(x,x')+d(y,y')  &\text{ if }x \in \mathcal S_{x'}, y \in \mathcal S_{y'}, x', y' \in \widetilde{\mathcal R}_1, x' \neq y',\\
d_{\widetilde{\mathcal R}_1}(x,y')+d(y,y') &\text{ if } x \in \widetilde{\mathcal R}_1, y \in \mathcal S_{y'}, y' \in \widetilde{\mathcal R}_1, \\
d(x,y) &\text{ if } x,y \in \mathcal S_{x'}, x' \in \widetilde{\mathcal R}_1. \end{cases} \label{distancemc}
\end{equation} 
Furthermore, define the mass measure $\tilde{\mu}_{1}$ on $(\widetilde{\mathcal T}_1,\tilde{d}_1)$ by
\begin{equation} \tilde{\mu}_{1}(\overline{\mathcal T}_x):=  \begin{cases} \sum \limits_{y \in \overline{\mathcal T}_x\cap \widetilde{\mathcal R}_1} \mu(\mathcal S_y) & \text{ if } x \in \widetilde{\mathcal R}_1, \\
\mu(\overline{\mathcal T}_x \cap \mathcal S_{x'}) & \text{ if } x \in \mathcal S_{x'}, x' \in \widetilde{\mathcal R}_1, \end{cases}  \label{muonemc}
\end{equation}
where $\overline{\mathcal T}_x:=\{y \in \widetilde{\mathcal T}_1: x \in [[\rho, y]]\}$. Note that \eqref{muonemc} uniquely identifies $\tilde{\mu}_1$. 

We refer to Figure \ref{transrule} for an illustration of the transition rule of the BMMC.
\begin{remark} In fact, the merging operation $\biguplus$ was only defined for intervals. However, each 
$E_i=[[\rho_i, \rho_i']]$, $i \in [3]$, is isometric to the interval $[0, d(\rho_i, \rho_i')]$, i.e. we write  $\biguplus_{i \in [k]} E_i$ to mean that we apply the merging operation to the strings of beads that we obtain when we consider the intervals $[0, d(\rho_i, \rho_i')]$ with mass measure pushed forward via the choice of class representatives.
\end{remark}

\begin{figure}[htb] \centering \def\svgwidth{\columnwidth} 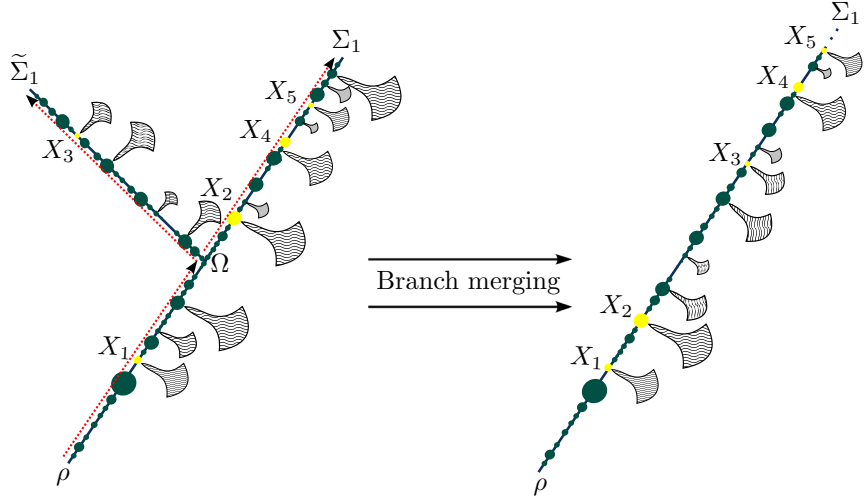  \caption{Transition rule of the BMMC. Reduced tree structure with subtrees attached in the atoms. Strings of beads are merged and subtrees are carried forward.} \label{pic4} \label{transrule} \end{figure} 

\begin{lemma} \label{BCRTstat} The distribution of the Brownian CRT on the space of continuum trees is a stationary distribution of the Branch Merging Markov Chain $(\mathcal B_k, k \geq 0)$.
\end{lemma}

\begin{proof}
Recall that, by Lemma \ref{exchange}, leaf labelling in the $(1/2, 1/2)$-model is exchangeable, i.e. we can embed $\mathcal R_2^{1/2,1/2}$ by sampling two leaves independently from the mass measure. The distribution of the mass split on the three branches of the Y-shaped tree $\mathcal R_2^{1/2,1/2}$ is {\rm Dirichlet}$(1/2,1/2,1/2)$, and the three branches are rescaled independent $(1/2,1/2)$-strings of beads when equipped with the restricted mass measure (Proposition \ref{cointoss}).
We merge the three $(1/2,1/2)$-strings of beads to obtain a $(1/2,3/2)$-string of beads (see Theorem \ref{Mainresult}). By the spinal decomposition theorem, the subtrees rooted on $\mathcal R_2^{1/2,1/2}$ which we carry forward, are rescaled independent Brownian CRTs, i.e. by Theorem \ref{embredtrees2} and Corollary \ref{BCRTScal}, the merged tree is a Brownian CRT. \end{proof}

We now study a discrete analogue of the BMMC, and use the framework of $\mathbb R$-trees with unit edge lengths. For $n \in \mathbb N$, let $\mathbb T_n^o$ be the set of rooted unlabelled trees with $n$ leaves and no degree-two vertices (except the root), and let $\mathbb T_n^{b,o}$ be the space of binary rooted unlabelled trees with $n$ leaves and no degree-two vertices, i.e. $\mathbb T_n^{b,o} \subset \mathbb T_n^{o}$. For $T \in \mathbb T_n^o$, we consider the uniform probability measure on $\mathcal L(T)$ by assigning mass $1/n$ to any of the leaves. Let $(\tau_k, k \geq 0)$ be a time-homogeneous Markov chain operating on $\mathbb T_n^o$ with the following transition rule. Conditionally given $\tau_0=T$ with root $\rho$, perform the following steps to obtain $\tau_1=\widetilde{T}$.
\begin{enumerate}
\item[(i)$^{\text{d}}$] Select one of the edges of $T$ uniformly at random, and insert one leaf at this edge to create the tree $T_{+1} \in \mathbb T_{n+1}^o$ according to the growth rule for uniform binary tree growth processes (i.e. as in the $(1/2,1/2)$-tree growth process), see Section \ref{alphathetamodel}.
\item[(ii)$^{\text{d}}$] Sample two distinct leaves $\Sigma_1$ and $\widetilde{\Sigma}_1$ from the empirical mass measure on the set of leaves $\mathcal L(T_{+1})$ of $T_{+1}$, as in (i)$^{\text{MC}}$, and let $\Omega$ be the unique branch point of the reduced tree $\mathcal R(T_{+1}, \Sigma_1, \widetilde{\Sigma}_1)$.
\item[(iii)$^{\text{d}}$] Sample leaf labels $2,\ldots,n$ according to a uniform allocation onto the $n-1$ unlabelled leaves (in the sense that leaves are sampled uniformly without replacement). Define the set of cut points as in (iii)$^{\text{MC}}$, and refer to the labelled tree as $T_{+1}^l$.
\item[(iv)$^{\text{d}}$] Consider $\mathcal R_2=\mathcal R(T_{+1}^l, \Sigma_1, \widetilde{\Sigma}_1)$ and the pairs $(E_i, \mu_{\mathcal R_2}\restriction_{E_i})$, $i\in[3]$, where
\begin{equation*} E_1=[[\rho, \Omega[[, \quad E_2=]]\Omega, \Sigma_1[[, \quad E_3=[[\Omega,\widetilde{\Sigma}_1[[.
\end{equation*}
Define the merged spine $\widetilde{\mathcal R}_1$ and the tree $\widetilde{\mathcal T}_1$ as in (iii)$^{\text{MC}}$ and (iv)$^{\text{MC}}$ to obtain the output tree $\widetilde{T}=\widetilde{\mathcal T}_1$.

\end{enumerate}

Note that both in the BMMC and in the discrete version, we accept $\mu_{\mathcal R_2}(\Omega) > 0$. In this case, when the first part of $E_2$ is not the first part of the merged string of beads, in step (iv)$^{\text{d}}$, we join the two atoms $\Omega$ and $x$ for some $x \in \mathcal R_2$, and form a larger atom with mass $\mu_{\mathcal R_2}(\Omega)+\mu_{\mathcal R_2}(x)$, see the convention from Remark \ref{closedintervals}. Furthermore, when subtrees are carried forward, we join the trees $\mathcal S_{\Omega}=(\pi^{\mathcal R_2})^{-1}(\Omega)$ and $\mathcal S_{x}=(\pi^{\mathcal R_2})^{-1}(x)$ at their root vertices. We treat the case $\mu_{\mathcal R_2}(\rho) > 0$ analogously. Note that $\widetilde{T}\in \mathbb{T}_n^o$ as we loose the leaf $\widetilde{\Sigma}_1$ under branch merging.

\begin{theorem} \label{discrete}
Let $n \geq 2$, and consider the Markov chain $(\tau_k, k \geq 0)$ operating on $\mathbb T_n^o$, with transition rule given by {\rm(i)}$^{\rm{d}}$ - {\rm(iv)}$^{\rm{d}}$. The space $\mathbb T_n^{b,o}$ is the only closed communicating class of $(\tau_k, k \geq 1)$, $(\tau_k, k \geq 0)$ has a unique stationary distribution $\pi$ which is supported by $\mathbb T_n^{b,o}$, and the distribution of $\tau_k$ converges to $\pi$ as $k \rightarrow \infty$.
\end{theorem}

\begin{proof} Consider two trees $T, \widetilde{T} \in \mathbb{T}_n^o$. We have to show the following.

\begin{itemize}
\item If $T \in \mathbb{T}_n^{b,o}$, then
\begin{equation}
\mathbb P \left(\tau_k \in \mathbb{T}_n^{b,o} \text{ for all } k \geq 1 | \tau_0=T\right) =1. \label{comclass3} \end{equation}
\item If $T, \widetilde{T} \in \mathbb{T}_n^{b,o}$, then
\begin{equation}
\mathbb P \left(\tau_k=\widetilde{T} \text{ for some } k \geq 1 | \tau_0=T\right) > 0. \label{comclass1} \end{equation}
\item If $T \in \mathbb{T}_n^o \setminus \mathbb{T}_n^{b,o}$, then
\begin{equation}
\mathbb P \left(\tau_k \in \mathbb{T}_n^{b,o} \text{ for some } k \geq 2 | \tau_0=T\right)  >0. \label{comclass2} \end{equation} \end{itemize}

We use an induction on the number of leaves $n$. For $n=2$, we have two possible trees, a Y-shaped tree (say $T_1$), and a tree with two leaves which are directly connected to the root (say $T_2$). Obviously, given $\tau_0=T_1$, $\tau_1=T_1$ with probability $1$. On the other hand, to obtain $T_1$ from $T_2$, after the leaf insertion, assign label $\widetilde{\Sigma}_1$ to the leaf directly connected to the root. This leaf gets lost in the branch merging operation, and we obtain $T_1$.

Now, assume that \eqref{comclass3}-\eqref{comclass2} hold for any $1 \leq m \leq n-1$ where $n \geq 2$. 
\begin{itemize}
\item \textit{Proof of \eqref{comclass3}}. \eqref{comclass3} is clear due to the nature of the leaf insertion procedure: only edges can be selected (but no branch points), and, in the binary case, there are no atoms in branch points.  Hence, degrees of branch points cannnot increase, and $\eqref{comclass3}$ follows.

\item \textit{Proof of \eqref{comclass1}}. Conditionally given $T \in \mathbb T_{n}^{b,o}$, we consider the two subtrees $T_1$ and $T_2$ into which $T$ splits at the first branch point. Let $n_1$ be the number of leaves in $T_1$ (i.e. $T_2$ has $n-n_1$ leaves). By the induction assumption, for any $T_1' \in \mathbb T_{n_1}^{b,o}$ and any $T_2' \in \mathbb T^{b,o}_{n-n_1}$ it holds that
\begin{align*}
\mathbb P \left(\tau_k={T}_i' \text{ for some } k \geq 1 | \tau_0=T_i\right) &> 0, \quad i=1,2.\end{align*}
Assigning label $2$ to a leaf in $T_2$, and restricting the remaining leaf insertion and sampling procedure for the leaves $\Sigma_1, \widetilde{\Sigma}_1$ and the labels $3,\ldots, n_1+1$ to the tree $T_1$ (and swapping the roles of $T_1$ and $T_2$ afterwards), we obtain positive probability to attain any tree $T' \in \mathbb T_{n}^{b,o}$ which splits at the first branch point into two subtrees $T_1'$ and $T_2'$ of sizes $n_1$ and $n-n_1$, respectively.

It remains to show that, for any $1 \leq n_1 \leq n-1$, we have positive probability of attaining a tree $T' \in \mathbb T_{n}^{b,o}$ which splits into two subtress of sizes $n_1+1$ and $n-n_1-1$ at the first branch point. This is easy to see, e.g. consider the case when the leaf insertion takes place in the subtree $T_1$, label 2 is assigned to a leaf in $T_1$, and $\Sigma_1$, $\widetilde{\Sigma}_1$ as well as labels $3, \ldots, n-n_1$ are all in $T_2$. This shows \eqref{comclass1}.

\item \textit{Proof of \eqref{comclass2}}. For a tree $T \in \mathbb{T}_n^{o} \setminus \mathbb{T}_n^{b,o}$, conditionally given that the degree of the root is $k$, consider the subtrees $S_1, \ldots, S_k$ of sizes $n_1, \ldots, n_k$ in which $T$ splits at the root (taking one edge each as a root edge for $S_i$, $1 \leq i \leq k$). 
First, we want to show that we have positive probability to decrease the degree of the root vertex by one. Therefore, apply the following procedure. In (i)$^{\text{d}}$ insert the leaf in $S_k$ to increase its size by one, assign label $2$ to one of the subtrees $S_{2}, \ldots, S_{k}$, and consider the subtree $S_1$ in the following.
\begin{enumerate}
\item[-] At the first branch point $\Omega_1$ of $S_1$, the tree splits into $k_1$ subtrees, say, $S_{1,1}, \ldots, S_{1,k_1}$. Assign labels $\Sigma_1, \widetilde{\Sigma}_1$ to leaves in $S_{1,1}$, $S_{1,k_1}$, respectively.
\item[-] Assign label $3$ to a leaf in one of the subtrees $S_{1,2}, \ldots, S_{1,k_1-1}$.
\item[-] Perform branch merging to obtain a tree $S_1'$ from $S_1$ whose size has decreased by one, as leaf $\widetilde{\Sigma}_1$ gets lost under branch merging.
\end{enumerate}
After branch merging, the tree $T$ becomes a tree $T' \in \mathbb T_n^{o}$, which splits at the root into subtrees $S_1', \ldots, S_k'$ of sizes $n_1-1, n_2, \ldots, n_{k-1}, n_k+1$, respectively. We perform this procedure $n_1-1$ times, and obtain a tree $T''$, which splits into subtrees $S_1'', \ldots, S_k''$ of sizes $1, n_2, \ldots, n_{k-1}, n_k+n_1-1$. Note that $S_1''$ consists of one single edge and one leaf. Now, perform the leaf insertion in $S_k''$, assign label $\Sigma_1$ to a leaf in $S_k''$, label $\widetilde{\Sigma}_1$ to the single leaf in $S_1''$, and label $2$ to a leaf in one of the subtrees $S_2'', \ldots, S_{k-1}''$. After branch merging, we only have a split into $k-1$ subtrees at the root edge. Clearly, we can perform this procedure another $k-2$ times, and, eventually, with positive probability, we obtain a tree with a root of degree $1$, which we denote by $\overline{T}$. The tree $\overline{T}$ has one edge connecting the root and a branch point at which the tree splits into $m$, say, subtrees $\overline{S}_1, \ldots, \overline{S}_m$. 

Next, we want to show that there is positive probability to obtain a tree from $\overline{T}$ with root degree one and first branch point degree two. Therefore, repeat the procedure as above $m-2$ times with $\overline{S}_1, \ldots, \overline{S}_m$. Eventually, with positive probability, we obtain a tree with a binary branch point, and denote the two subtrees of sizes $n_1$ and $n-n_1$ by $\overline{T}_1$ and $\overline{T}_2$, respectively. 

We now use the induction assumption related to \eqref{comclass2}. Assigning label $2$ to a leaf in $\overline{T}_2$, and restricting the remaining leaf insertion and sampling procedure for the leaves $\Sigma_1, \widetilde{\Sigma}_1$ and the labels $3,\ldots, n_1+1$ to $\overline{T}_1$ (and swapping the roles of $\overline{T}_1$ and $\overline{T}_2$ afterwards), by the induction assumption, we obtain positive probability to turn $\overline{T}_1$ and $\overline{T}_2$ into binary trees, and hence, positive probability to turn $\overline{T}$ into a binary tree, i.e. \eqref{comclass2} follows.
\end{itemize}

Hence, for any $n \in \mathbb N$, the Markov chain $(\tau_k, k\geq 0)$ on $\mathbb T_n^o$ has exactly one communicating class, namely $\mathbb T_n^{b,o}$. For $T \in \mathbb T_n^0$, $\mathbb P(\tau_1=T\mid \tau_0=T)>0$, i.e. the chain is aperiodic. We conclude that $(\tau_k, k\geq 0)$ has a unique stationary distribution $\pi$, supported on $\mathbb T_n^{b,o}$, and the distribution of $\tau_k$ converges to $\pi$ as $k \rightarrow \infty$.
\end{proof}

Theorem \ref{discrete} gives evidence that the Brownian CRT should be the \textit{unique} stationary distribution of the Branch Merging Markov Chain.

\begin{conjecture} \label{conj} The Markov chain $(\mathcal B_k, k \geq 0)$ operating on the space of continuum trees has a unique stationary distribution given by the Brownian CRT, and the distribution of $\mathcal B_k$ converges to the distribution of the Brownian CRT as $k \rightarrow \infty$ with respect to the Gromov-Hausdorff topology (up to scaling distances by a constant).
\end{conjecture}

\begin{remark} 
Note that the $\mathbb R$-trees we consider are compact, and hence, continuum trees are of finite height. For $\epsilon >0$ fixed, any continuum tree $\mathcal B_0:=\mathcal T$ and any branch point of $\mathcal B_0$, we consider the subtrees attached of height $\geq \epsilon$. The probability that, in the leaf sampling procedure, the two leaves $\Sigma_1$ and $\widetilde{\Sigma}_1$ are in two distinct subtrees of this branch point is positive, and the number of subtrees with height $\geq \epsilon$ is reduced by one after the branch merging transition. Eventually, the BMMC will turn this branch point into a binary branch point (ignoring subtrees of size $< \epsilon$). Subtrees of size $<\epsilon$ do not matter for convergence in the Gromov-Hausdorff sense, and hence we conjecture the convergence in distribution of the BMMC $(\mathcal B_k, k \geq 0)$ to the Brownian CRT.

The process $(\mathcal B_k, k \geq 0)$ evolves in the space of continuum trees, an uncountable and awkward state space for a Markov process. This makes it hard to provide a rigorous proof of Conjecture \ref{conj}.

\end{remark}

\subsection{General branch merging and the leaf embedding problem} \label{genbm}

We now turn to a more general branch merging setting where leaf sampling is not based on uniform sampling from the mass measure as in the BMMC. As a tree analogue of merging of $(\alpha, \theta_i)$-strings of beads, $i \in [k]$, in the special case when $k=3$ and $\theta_1=\theta_2=1-\alpha$, $\theta_3=\alpha$, we introduce branch merging on Ford CRTs, and give two applications. 

This procedure will be more involved than the branch merging algorithm in the BMMC. In order to define branch merging on general Ford CRTs, using our merging algorithm for strings of beads, we need a {\rm Dirichlet}$(\alpha, 1-\alpha, 1-\alpha)$ mass split between the three branches of the reduced tree structure. In general, this cannot be achieved via uniform sampling of leaves from the mass measure, see e.g. \citep{1}, Section 4.2.
Therefore, we have to apply the $(\alpha,1-\alpha)$-coin tossing construction in $\mathcal T^{\alpha, 1-\alpha}$ (which corresponds to uniform sampling when $\alpha=1/2$) to determine the start configuration. A similar generalisation is needed for the partition sampling procedure in $\text{(ii)}^{\text{MC}}$.

It is a straightforward generalisation of Lemma \ref{BCRTstat} to conclude that Ford CRTs are distributionally invariant under this general procedure. However, our goal is to establish a coupling between the sequences of reduced trees related to the $(\alpha,1-\alpha)$- and the $(\alpha,2-\alpha)$-model. The simple procedure involving two sampled leaves is enough to relate $\mathcal R_1^{\alpha,1-\alpha}$ and $\mathcal R_1^{\alpha,2-\alpha}$ but not to apply the procedure recursively, and to couple the sequences $(\mathcal R_k^{\alpha,1-\alpha}, k \geq 1)$ and $(\mathcal R_k^{\alpha,2-\alpha}, k \geq 1)$. It is necessary to find the bead of the second leaf on $\mathcal R_1^{\alpha,2-\alpha}$ to obtain a {\rm Dirichlet}$(\alpha,1-\alpha, 2-\alpha)$ mass split at the atom of label $2$ in $\widetilde{\mathcal R}_1$ which identifies the subtree we have to select for the recursive application of the merging procedure. Therefore, it is crucial for the recursive leaf embedding procedure to obtain successively finer allocations of the non-embedded leaf labels, that is, after step $k \geq 1$, the labels $\{k+1, k+2, \ldots\}$, onto the atoms of the mass measure $\mu_{\widetilde{\mathcal R}_k}$ on the reduced tree $\widetilde{\mathcal R}_k$.

\subsubsection{Branch merging on Ford CRTs} \label{BMFORD}

To highlight the generalisations described above, we first introduce the branch merging algorithm needed to embed the leaf labels of the $(1/2,3/2)$-tree growth process in the Brownian CRT. In this procedure, the $(\alpha,1-\alpha)$-coin tossing construction is not needed as, for $\alpha=1/2$, this reduces to uniform sampling from the mass measure.

In the Brownian case, by Lemma \ref{exchange}, leaf labelling in the $(1/2,1/2)$-tree growth model is exchangeable, and hence, leaf embedding can be done via uniform sampling from the mass measure. On the other hand, the Brownian CRT is the scaling limit of the $(1/2,3/2)$-tree growth process in which leaf labelling is non-exchangeable. 

We construct a Brownian CRT with leaf $1$ embedded according to the $(1/2,3/2)$-growth rule from a Brownian CRT equipped with the natural mass measure. In this sense, uniform sampling from the mass measure on the Brownian CRT in combination with branch merging allows to embed the first leaf of the $(1/2, 3/2)$-model by using a much simpler approach than the  $(1/2,3/2)$-coin tossing construction from Lemma \ref{leafemblem}. In Section \ref{recbm} we will explain the recursive leaf embedding algorithm which allows us to obtain a Brownian CRT with all leaves of the $(1/2,3/2)$-model embedded.

The leaf embedding algorithm can be directly obtained from Definition \ref{defbmmc}, in particular by generalising $\rm{(ii)}^{\rm MC}$ in the sense that we use an i.i.d. allocation of the leaf labels $\mathbb N_{\geq 3}$ to determine the cut points on the reduced tree structure.

\begin{algorithm}[$(1/2, 3/2)$-leaf embedding in the Brownian CRT] \label{bmalgobrown} Let $(\mathcal T, d, \rho, \mu)$ be a Brownian CRT with root $\rho$, mass measure $\mu$ and metric $d \colon \mathcal T \times \mathcal T \rightarrow \mathbb R_0^+$.

\begin{enumerate}
\item[${\rm (i)}^{\rm B}$] \textit{Start configuration}.\label{step1} Sample three leaves $\Omega_0, \Omega_1, \Omega_2$ independently from $\mu$, and denote the \textit{first} branch point of the spines from $\rho$ to $\Omega_0,\Omega_1, \Omega_2$ by $\Omega$. We obtain the three branches
\begin{equation*} \widetilde{E}_0:=]]{\Omega}, \Omega_0]], \quad \widetilde{E}_1:=]]{\Omega}, \Omega_1]], \quad \widetilde{E}_2:=]]{\Omega}, \Omega_2]], \end{equation*}
and define the start configuration $(\Sigma_1, \widetilde{\Sigma}_1, \Theta)$ by
\begin{equation*} \left(\Sigma_1, \widetilde{\Sigma}_1, \Theta\right):= \begin{cases} \left(\Omega_2, \Omega_0, \Omega_1 \right) &\text{ if } \widetilde{E}_0 \cap \widetilde{E}_1 \neq \emptyset, \\
\left(\Omega_1, \Omega_0, \Omega_2\right) &\text{ if } \widetilde{E}_0 \cap \widetilde{E}_2 \neq \emptyset, \\
\left(\Omega_0, \Omega_1, \Omega_2 \right) &\text{ if } \widetilde{E}_1 \cap \widetilde{E}_2 \neq \emptyset,
\end{cases}
\end{equation*}
i.e. we choose $(\Sigma_1, \widetilde{\Sigma}_1, \Theta)$ such that $\Omega$ is connected to the root and the leaf $\Sigma_1$ by one edge each, and to the two leaves $\widetilde{\Sigma}_1$ and $\Theta$ via another branch point $\rho_\Theta:=\pi^{\mathcal R_2}(\Theta)$.

Denote by $\mathcal R_2=\mathcal R(\mathcal T, \Sigma_1, \widetilde{\Sigma}_1)$ the reduced tree spanned by $\rho$, $\Sigma_1$ and $\widetilde{\Sigma}_1$, and by $\mathcal S_{x}:=(\pi^{\mathcal R_2})^{-1}(x)$ the subtrees of $\mathcal T$ rooted at  $x \in \mathcal R_2$, that is, the connected components of $\mathcal T \setminus \mathcal R_2$ completed by their root vertices on $\mathcal R_2$. We equip $\mathcal R_2$ with the mass measure $\mu_{\mathcal R_2}$ capturing projected subtree massees, i.e. $\mu_{\mathcal R_2}(x)=\mu(\mathcal S_x)$, $x \in \mathcal R_2$. As illustrated in Figure \ref{pic3}, we obtain a split of $\mathcal R_2$ into the four branches
\begin{equation} E_0:=]]\rho, {\Omega}[[,\quad E_1:=]]{\Omega}, \Sigma_1[[,\quad E_{2}:=]]{\Omega}, \rho_\Theta[[, \quad E_{3}:=]]\rho_\Theta, \widetilde{\Sigma}_1[[. \label{branches2} \end{equation} 

\item[${\rm (ii)}^{\rm B}$] \textit{Partition and cut point sampling}. \label{step2} Determine a partition $\{\Pi_x, x \in \mathcal R_2\}$ of the label set $\mathbb N_{\geq 2}=\{2,3,\ldots\}$ via $\Pi_x =\{k: Y_k=x\}$ where $Y_2 =\rho_{\Theta}$ and $(Y_k)_{k \geq 3}$ are i.i.d. with distribution $\mu_{\mathcal R_2}$.

Define a sequence of cut points $(X_k)_{k \geq 1}$ on $E_i$, $i \in [3]$,  as follows.
\begin{itemize} \item Let $E_i^{(1)}=E_i$, $i \in [3]$, and set $J_0=0$. 
\item For $k \geq 1$, conditionally given $E_i^{(k)}, i \in [k]$,
let $$X_{k}=Y_{J_k}, \quad J_k:=\inf \{ j \geq J_{k-1}+1: Y_j \in E^{(k)} \}$$ where $E^{(k)}=\bigcup_{i \in [3]} E_i^{(k)}$. For $i \in [3]$, set $$E_i^{(k+1)} = \begin{cases}E_i^{(k)} \setminus [[\rho_i, X_k]]& \text{ if }X_k \in E_i^{(k)}, \\  E_i^{(k)}  &\text{ if }X_k \notin E_i^{(k)}, \end{cases}$$  
where $\rho_i$ denotes the left endpoint of $E_i$ in \eqref{branches2}, $i \in [3]$, respectively.
\end{itemize}

\item[${\rm (iii)}^{\rm B}$] \textit{Tree pruning and spine merging}. \label{step3} Merge the strings of beads $(E_i, \mu_{\mathcal R_2}\restriction_{E_i})$, $i\in[3
]$, based on the set of cut points $\Upsilon=\{X_k, k\geq 1\}$, and define the merged branch $\widetilde{\mathcal R}_1$ connecting the root $\rho$ and the leaf $\Sigma_1$ via \begin{equation} \widetilde{\mathcal R}_1:=[[\rho, {\Omega}[[  \cup \{\rho_{\Theta}\} \cup \left(\biguplus \limits_{i\in[3]} E_i\right) \cup \{\Sigma_1\}, \label{defr1} \end{equation}
where the operator $\biguplus$ was defined in \eqref{operator}, Section \ref{mergalg}. We obtain the metric $d_{\widetilde{\mathcal R}_1}$ on $\widetilde{\mathcal R}_1$ by aligning the parts of $\widetilde{\mathcal R}_1$ in the order given in \eqref{defr1}.

\item[${\rm (iv)}^{\rm B}$] \textit{Subtree replanting}. \label{step4} 
Attach the trees $\mathcal S_x$ to their root vertices $x \in \widetilde{\mathcal R}_1$, to obtain a tree $(\widetilde{\mathcal T}_1,\tilde{d}_1)$, i.e. $\widetilde{\mathcal T}_1=\widetilde{\mathcal R}_1 \cup \bigcup _{x \in \widetilde{\mathcal R}_1} \mathcal S_x$, and equip $\widetilde{\mathcal T}_1$ with the mass measure $\tilde{\mu}_1$ pushed forward via these operations.
\end{enumerate}
\end{algorithm}

We provide explicit definitions of ${\rm (iii)}^{\rm B}$ and ${\rm (iv)}^{\rm B}$. The metric $d_{\widetilde{\mathcal R}_1}: \widetilde{\mathcal R}_1 \times \widetilde{\mathcal R}_1 \rightarrow \mathbb R_0^{+}$ on $\widetilde{\mathcal R}_1$ can be defined by $d_{\widetilde{\mathcal R}_1}(x,y):=\lvert d_{\widetilde{\mathcal R}_1}(\rho, x)-d_{\widetilde{\mathcal R}_1}(\rho, y)\rvert$ where
\begin{equation*}
d_{\widetilde{\mathcal R}_1}(\rho,x):=\begin{cases} d(\rho,x) & \text{ if } x \in [[\rho, \Omega[[,\\
d(\rho, \Omega) &\text{ if } x=\rho_\Theta, \\
d(\rho, \Omega) + d_{\biguplus}(\rho_\Theta, x) &\text{ if } x \in E_1 \cup E_{2} \cup E_{3}, \\
d(\rho, \Omega) + d(\rho_{\Theta}, \Omega) + d(\rho_{\Theta}, \widetilde{\Sigma}_1)+ d(\Omega, {\Sigma}_1) &\text{ if } x=\Sigma_1. \end{cases}
\end{equation*}

The metric space $(\widetilde{\mathcal T}_1,\tilde{d}_1)$ is given by $\widetilde{\mathcal T}_1:=\widetilde{\mathcal R}_1 \cup \bigcup _{x \in \widetilde{\mathcal R}_1} \mathcal S_x$ and $\tilde{d}_1: \widetilde{\mathcal T}_1 \times \widetilde{\mathcal T}_1  \rightarrow \mathbb R_{0}^+$,
\begin{equation}
\tilde{d}_1(x,y):= \begin{cases} d_{\widetilde{\mathcal R}_1}(x,y) & \text{ if } x, y \in \widetilde{\mathcal R}_1,\\
 d_{\widetilde{\mathcal R}_1}(x',y')+d(x,x')+d(y,y')  &\text{ if }x \in \mathcal S_{x'}, y \in \mathcal S_{y'},  x', y' \in \widetilde{\mathcal R}_1, x' \neq y',\\
d_{\widetilde{\mathcal R}_1}(x,y')+d(y,y') &\text{ if } x \in \tilde{\mathcal R}_1, y \in \mathcal S_{y'},  y' \in \widetilde{\mathcal R}_1, \\
d(x,y) &\text{ if } x,y \in \mathcal S_{x'},  x' \in \widetilde{\mathcal R}_1. \end{cases} \label{distance}
\end{equation} 
The mass measure $\tilde{\mu}_{1}$ on $(\widetilde{\mathcal T}_1,\tilde{d}_1)$ is obtained via
\begin{equation} \tilde{\mu}_{1}(\overline{\mathcal T}_x):=  \begin{cases} \sum _{y \in \overline{\mathcal T}_x\cap \widetilde{\mathcal R}_1} \mu(\mathcal S_y) & \text{ if } x \in \widetilde{\mathcal R}_1, \\
\mu(\overline{\mathcal T}_x \cap \mathcal S_{x'}) & \text{ if } x \in \mathcal S_{x'}, x' \in \widetilde{\mathcal R}_1, \end{cases}  \label{muone}
\end{equation}
where $\overline{\mathcal T}_x:=\{y \in \widetilde{\mathcal T}_1: x \in [[\rho, y]]\}$. Note that \eqref{muone} uniquely identifies $\tilde{\mu}_1$.

\vspace{0.3cm}

We now present the branch merging algorithm on Ford CRTs $\mathcal T^{\alpha,1-\alpha}$, $\alpha \in (0,1)$, in full generality. We pick three leaves according to $(\alpha,1-\alpha)$-coin tossing, and then choose the cut points similarly to the transition rule presented in Algorithm \ref{bmalgobrown}. External edges determine $(\alpha, 1-\alpha)$-strings of beads, where internal edges are related to $(\alpha, \alpha)$-strings of beads. Therefore, on external edges, atoms have to be selected according to $(\alpha, 1-\alpha)$-coin tossing sampling; on internal edges we use uniform sampling from the mass measure which corresponds to $(\alpha, \alpha)$-coin tossing sampling. 

Algorithm \ref{bmalgobrown} is the special case $\alpha=1/2$ of the following algorithm.

\begin{algorithm}[Branch merging on Ford CRTs]\label{bmalgo} Let $(\mathcal T^{\alpha, 1-\alpha}, d, \rho, \mu)$ be a Ford CRT with parameter $\alpha \in (0,1)$, root $\rho$, mass measure $\mu$ and metric $d \colon \mathcal T \times \mathcal T \rightarrow \mathbb R_0^+$.

\begin{enumerate}
\item[(i)] \textit{Start configuration}.\label{step1} Embed three leaves $\Omega_0, \Omega_1, \Omega_2$, i.e. the tree $\mathcal R_3^{\alpha, 1-\alpha}$, into $\mathcal T^{\alpha,1-\alpha}$ by using the $(\alpha, 1-\alpha)$-coin tossing construction. 

Define the start configuration $(\Sigma_1, \widetilde{\Sigma}_1, \Theta)$, the reduced tree $(\mathcal R_2, \mu_{\mathcal R_2})$ and the subtrees $\mathcal S_x, x \in \mathcal R_2$, as in Algorithm \ref{bmalgobrown}${\rm (i)}^{\rm B}$.

\item[(ii)] \textit{Partition and cut point sampling}. \label{step2} Determine a partition $\{\Pi_x, x \in \mathcal R_2\}$ of the label set $\mathbb N_{\geq 2}=\{2,3,\ldots\}$ via $\Pi_x =\{k: Y_k=x\}$ where the sequence of random variables $(Y_k)_{k \geq 2}$ is determined as follows. Set $Y_2 = \rho_{\Theta}$, and, for $k \geq 3$, select one of the branches $E_i$, $i\in[3]$, of $\mathcal R_2$ or the atom $\rho_{\Theta}$ proportionally to mass assigned by $\mu_{\mathcal R_2}$. If $\rho_{\Theta}$ is selected, set $Y_k=\rho_{\Theta}$. Otherwise, if a branch $E_i$ is selected, proceed as follows;

\begin{itemize} \item if $E_i$ is internal, i.e. $i \in \{0,2\}$, sample $Y_k \in E_i$ from the normalised mass measure $\mu_{\mathcal R_2}(E_i)^{-1}\mu_{\mathcal R_2} \restriction_{E_i}$; 
\item if $E_i$ is external, i.e. $i\in \{1,3\}$, select the top part $]]Y_{i,k}, \Sigma_k^*[[$ or the bottom part $E_i \setminus]]Y_{i,k}, \Sigma_k^*[[$ of $E_i$ proportionally to mass where $$Y_{i,k}:={\rm argmax}_{2 \leq j \leq k-1: Y_j \in E_i}d(\rho,Y_j)$$ is the element $Y_j \in E_i$ closest to the leaf $\Sigma_k^*$ in $E_i$ where $\Sigma_k^*=\Sigma_1$ if $i=1$ and $\Sigma_k^*=\widetilde{\Sigma}_1$ if $i=3$; if the top part is selected, perform $(\alpha, 1-\alpha)$-coin tossing sampling on $(]]Y_{i,k}, \Sigma_k^*[[, \mu_{\mathcal R_2} \restriction_{]]Y_{i,k}, \Sigma_k^*[[})$ to determine $Y_k$; otherwise, sample $Y_k$ from the normalised mass measure on the bottom part $E_i \setminus ]]Y_{i,k}, \Sigma_k^*[[$. Note that, if $Y_j \notin E_i$ for all $2 \leq j \leq k-1$, then the top part of $E_i$ corresponds to the full branch $E_i$.
\end{itemize}

Determine the sequence of cut points $(X_k)_{k \geq 1}$ on the branches $E_i, i \in [3]$, as in Algorithm \ref{bmalgobrown}(ii)$^{\rm B}$.

\item[(iii)-(iv)] \textit{Tree pruning, spine merging and subtree replanting}. Define the merged spine $\widetilde{\mathcal R}_1$ and the tree $(\widetilde{\mathcal T}_1, \tilde{d}_1, \tilde{\mu}_1)$ as in Algorithm \ref{bmalgobrown}${\rm (iii)}^{\rm B}$-${\rm (iv)}^{\rm B}$. \end{enumerate}
\end{algorithm}

Note that the only step which was directly carried forward from the transition rule of the BMMC (i.e. without any generalisation) is (iv$)^{\text{MC}}$=${\rm (iv)}^{\rm B}$=(iv).

\begin{figure}[htb] \centering \def\svgwidth{6cm} 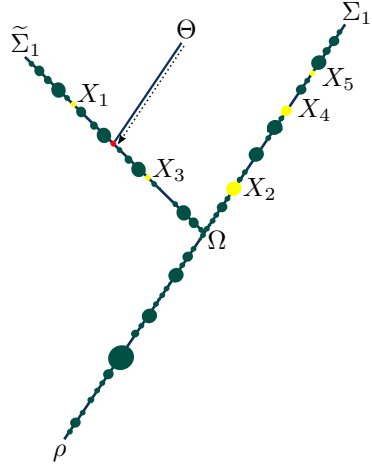  \caption{The reduced tree split into four branches with subtree masses as atoms. First five cut points $X_1, \ldots, X_5$ are displayed.} \label{pic3}\end{figure} 

\begin{theorem}[The merged tree] \label{basictransformation}
The quadruple $(\widetilde{\mathcal T}_1,\tilde{d}_1, \rho, \tilde{\mu}_1)$ constructed in Algorithm \ref{bmalgo}, see Figure \ref{pic4}, is a random $\mathbb R$-tree with the following properties.
\begin{enumerate}
\item[(i)] Let $\mu_{\widetilde{\mathcal R}_1}$ denote the discrete mass measure on $\widetilde{\mathcal R}_1$ obtained by assigning mass $\mu({\mathcal S}_x)$ to $x \in \widetilde{\mathcal R}_1$, where ${\mathcal S}_x$, $x \in \widetilde{\mathcal R}_1$, are the connected components of $\widetilde{\mathcal T}_1 \setminus \widetilde{\mathcal R}_1$, rooted at $x \in \widetilde{\mathcal T}_1$ respectively. 

Then $(\widetilde{\mathcal R}_1=[[\rho, \Sigma_1]],\mu_{\widetilde{\mathcal R}_1})$ is an $(\alpha, 2-\alpha)$-string of beads.

\item[(ii)] Given $(\widetilde{\mathcal R}_1,\mu_{\widetilde{\mathcal R}_1})$, the trees
\begin{equation*}
\left( {\mathcal S}_x, \tilde{\mu}_1({\mathcal S} _x)^{-\alpha} \tilde{d}_1 \restriction_{{\mathcal S}_x},\rho_{{\mathcal S}_x}:=x,  \tilde{\mu}_1({\mathcal S}_x)^{-1} \tilde{\mu}_1 \restriction_{{\mathcal S}_x} \right), \qquad x \in  \widetilde{\mathcal R}_1,
\end{equation*}
are isometric to independent copies of $(\mathcal T^{\alpha,1-\alpha}, d, \rho, \mu)$.
\item[(iii)] The tree $(\widetilde{\mathcal T}_1,\tilde{d}_1, \rho, \tilde{\mu}_1)$ is a Ford CRT with parameter $\alpha \in (0,1)$, i.e. 
\begin{equation*}
(\widetilde{\mathcal T}_1,\tilde{d}_1, \rho, \tilde{\mu}_1) \,{\buildrel d \over =}\, (\mathcal T^{\alpha, 1-\alpha}, d, \rho, \mu)\,{\buildrel d \over =}\, (\mathcal T^{\alpha, 2-\alpha}, d, \rho, \mu).\end{equation*}
\end{enumerate}
\end{theorem}

We prove Theorem \ref{basictransformation} in a series of propositions.

\begin{proposition} The leaf embedding procedure Algorithm \ref{bmalgo}{\rm (i)} induces the mass split 
\begin{align} &\left(\mu_{\mathcal R_2}(E_0), \mu_{\mathcal R_2}(E_1), \mu_{\mathcal R_2}(E_{2}), \mu_{\mathcal R_2}(E_{3}), \mu_{\mathcal R_2}(\mathcal \rho_\Theta) \right) \nonumber \\ &\hspace{4cm}\sim \text{\text{{\rm Dirichlet}}}\left(\alpha,1-\alpha,\alpha, 1-\alpha,1-\alpha\right).\label{embedmass} \end{align} Furthermore, 
\begin{equation} \left(\mu_{\mathcal R_2}(E_0), \mu_{\mathcal R_2}(\rho_\Theta), \mu_{\mathcal R_2}(E_1)+\mu_{\mathcal R_2}(E_{2})+\mu_{\mathcal R_2}(E_{3})\right)\sim \text{\text{{\rm Dirichlet}}}\left(\alpha, 1-\alpha, 2-\alpha\right).\label{aggmass} \end{equation}
\end{proposition}
\begin{proof} We refer to \citep{1}, Section 3.5. It was shown that the leaf embedding procedure on $\mathcal T^{\alpha,1-\alpha}$ yields a reduced tree $\mathcal R(\mathcal T^{\alpha,1-\alpha}, \Sigma_1, \widetilde{\Sigma}_1, \Theta)$ subject to a {\rm Dirichlet} mass split with parameter $\alpha$ for each internal branch, and $(1-\alpha)$ for each external branch. To see \eqref{embedmass} note that the mass of the branch $]]\rho_\Theta, \Theta[[$ of $\mathcal R(\mathcal T^{\alpha,1-\alpha}, \Sigma_1, \widetilde{\Sigma}_1, \Theta)$ corresponds to the mass of the atom $\rho_\Theta$ on $\mathcal R_2$. \eqref{aggmass} follows from the aggregation property of the {\rm Dirichlet} distribution, Proposition \ref{dirprop}(i).
\end{proof}

\begin{proposition} Let $\Pi_x$, $x \in \mathcal R_2$, be the label sets from Algorithm \ref{bmalgo}{\rm (ii)}. Furthermore, let $(\Sigma_k, k \geq 1)$ be a sequence of leaves as in Theorem \ref{embredtrees2}. Then, for any $x \in \mathcal R_2$ and $\mathcal S_x=(\pi^{\mathcal R_2})^{-1}(x)$, we have $\Pi_x \,{\buildrel d \over =}\, \{k \geq 2: \Sigma_k \in \mathcal S_x\}$.\end{proposition}

\begin{proof} This follows from Lemma \ref{embedall}. Given $\mathcal R(\mathcal T^{\alpha,1-\alpha}, \Sigma_1, \ldots, \Sigma_k)$ the selection rule for the atom $x$ which $\Sigma_{k+1}$ is attached to, is such that an edge of the reduced tree is selected proportionally to mass. If an external edge is selected, the $(\alpha,1-\alpha)$-coin tossing construction from Lemma \ref{leafemblem} is applied. If an internal edge is selected, an atom is chosen proportionally to weight. Projecting the labels of $\Sigma_k, k \geq 2$, onto $\mathcal R_2$ yields the claimed equality in distribution of the label sets, where we set $\Sigma_2=\Theta$. \end{proof}

\begin{proposition} \label{beads2}
Let $\mathcal A(\mu_{\mathcal R_2}):=\{x \in \mathcal R_2: \mu_{\mathcal R_2}(x)>0\}$ be the set of atoms of $\mu_{\mathcal R_2}$. Conditionally given $\mathcal A(\mu_{\mathcal R_2})$, the quadruples
\begin{equation*}
\left(\mathcal S_x,  \mu({\mathcal S_x})^{-\alpha} d \restriction_{\mathcal S_x}, \rho_{\mathcal S_x}:=x, \mu({\mathcal S_x})^{-1} \mu \restriction_{\mathcal S_x}\right), \quad x \in \mathcal A (\mu_{\mathcal R_2}),
\end{equation*}
are isometric to independent copies of $(\mathcal T^{\alpha,1-\alpha},d,\rho,\mu)$.  
\end{proposition}

\begin{proof} 
This follows directly from Theorem \ref{embredtrees2} in combination with Lemma \ref{leafemblem}.
\end{proof}

\begin{proposition}\label{sob}                                                                                                                         
The atom $\rho_{\Theta}$  splits $\widetilde{\mathcal R}_1$ into $]]\rho, \rho_{\Theta}[[$, $\rho_{\Theta}$, $]]\rho_\Theta, \Sigma_1[[$ such that
\begin{equation}
\left(\mu_{\widetilde{\mathcal R}_1}(]]\rho, \rho_{\Theta}[[), \mu_{\widetilde{\mathcal R}_1}(\rho_{\Theta}), \mu_{\widetilde{\mathcal R}_1}(]]\rho_\Theta, \Sigma_1[[)\right) \sim {\rm Dirichlet}(\alpha,1-\alpha,2-\alpha). \label{ronedir}
\end{equation}
Furthermore, the pair
\begin{equation} \left(\mu_{\widetilde{\mathcal R}_1}(]]\rho_\Theta, \Sigma_1[[)^{-\alpha}]]\rho_\Theta, \Sigma_1[[, \mu_{\widetilde{\mathcal R}_1}(]]\rho_\Theta, \Sigma_1[[)^{-1} \mu_{\widetilde{\mathcal R}_1}\restriction_{]]\rho_\Theta, \Sigma_1[[}\right)\label{ss1} \end{equation} is an $(\alpha, 2-\alpha)$-string of beads, and
the pair \begin{equation}
\left(\mu_{\widetilde{\mathcal R}_1}(]]\rho, \rho_\Theta[[)^{-\alpha}]]\rho, \rho_\Theta[[, \mu_{\widetilde{\mathcal R}_1}(]]\rho, \rho_\Theta[[)^{-1} \mu_{\widetilde{\mathcal R}_1}\restriction_{]]\rho, \rho_\Theta[[[}\right) \label{ss2}\end{equation} is an $(\alpha, \alpha)$-string of beads. The strings of beads \eqref{ss1} and \eqref{ss2} are independent of each other, and of the {\rm Dirichlet} mass split \eqref{ronedir}. 
\end{proposition}
\begin{proof} First of all, recall the definition of $E_i$, $i \in [3]$, i.e. 
\begin{equation*} E_1 = ]]\Omega, \Sigma_1[[, \quad E_2 =]]\Omega, \rho_{\Theta}[[, \quad E_3=]]\rho_{\Theta}, \widetilde{\Sigma}_1[[,
\end{equation*}
and note that, as a consequence of the embedding procedure, $\mu_{\widetilde{\mathcal R}_1}(E_1)^{-\alpha}E_i$, $i\in [3]$, equipped with $\mu_{\widetilde{\mathcal R}_1}(E_i)^{-1}\mu_{\tilde{\mathcal R}_1}\restriction_{E_i}$, $i \in [3]$, are independent $(\alpha, \theta_i)$-strings of beads, respectively, where $\theta_1=\theta_3=1-\alpha$, $\theta_2=\alpha$, see \citep{1}, Lemma \ref{cointoss}. Hence, by construction of 
$]]\rho_\Theta, \Sigma_1[[\subset \widetilde{\mathcal R}_1$ using the operator $\biguplus$ from \eqref{operator} and Theorem \ref{Mainresult}, the pair
\begin{equation*} (\mu_{\widetilde{\mathcal R}_1}(]]\rho_\Theta, \Sigma_1[[)^{-\alpha}]]\rho_\Theta, \Sigma_1[[,\mu_{\widetilde{\mathcal R}_1}(]]\rho_\Theta, \Sigma_1[[)^{-1}\mu_{\widetilde{\mathcal R}_1} \restriction_{]]\rho_\Theta, \Sigma_1[[}) \end{equation*} is an $(\alpha,2-\alpha)$-string of beads. On the other hand, the interval $E_0=]]\rho, \rho_\Theta[[$ with distance $d \restriction_{]]\rho, \rho_\Theta[[}$ is not affected by Algorithm \ref{bmalgo}, in particular $d \restriction_{]]\rho, \rho_\Theta[[}=d_{\widetilde{\mathcal R}_1} \restriction_{]]\rho, \rho_\Theta[[}$. Therefore, 
\begin{equation*} (\mu_{\widetilde{\mathcal R}_1}(]]\rho, \rho_\Theta[[)^{-\alpha} ]]\rho, \rho_\Theta[[, \mu_{\widetilde{\mathcal R}_1}(]]\rho, \rho_\Theta[[)^{-1} \mu_{\widetilde{\mathcal R}_1} \restriction_{]]\rho, \rho_\Theta[[} ) \end{equation*} is an $(\alpha, \alpha)$-string of beads, cf. Theorem \ref{embredtrees2}. By Proposition \ref{cointoss}, the strings of beads \eqref{ss1} and \eqref{ss2} are independent of each other, and of the mass split \eqref{ronedir}. \eqref{ronedir} is a direct consequence of the branch merging operation and \eqref{aggmass}.
\end{proof}

\begin{lemma} \label{Rsob}
The pair $(\widetilde{\mathcal R}_1,\mu_{\widetilde{\mathcal R}_1})$ is an $(\alpha, 2-\alpha)$-string of beads.
\end{lemma}

\begin{proof} By Proposition \ref{sob} we have a {\rm Dirichlet}($\alpha, 1-\alpha, 2-\alpha)$ mass split, an independent $(\alpha, \alpha)$- and an independent $(\alpha, 2-\alpha)$-string of beads as in Proposition  \ref{sobconstruc}, and we conclude that $(\widetilde{\mathcal R}_1,\mu_{\widetilde{\mathcal R}_1})$ is an $(\alpha, 2-\alpha)$-string of beads. \end{proof}

\begin{proof}[Proof of Theorem \ref{basictransformation}]
Lemma \ref{Rsob} proves Theorem \ref{basictransformation}(i). For (ii), note that, as an immediate consequence of the nature of the branch merging algorithm, the rescaled subtrees $\tilde{\mu}_1({\mathcal S}_x)^{-\alpha}\mathcal S_x = {\mu}_{\widetilde{\mathcal R}_1}(x)^{-\alpha}\mathcal S_x, x \in \widetilde{\mathcal R}_1$, are i.i.d. copies of $\mathcal T^{\alpha,1-\alpha}$, see Proposition \ref{beads2}. (iii) is a direct consequence of the spinal decomposition theorem for $\theta=2-\alpha$, cf. Lemma \ref{leafemblem}, and Corollary \ref{fords}. \end{proof}

\begin{figure}[htb] \centering \def\svgwidth{\columnwidth} 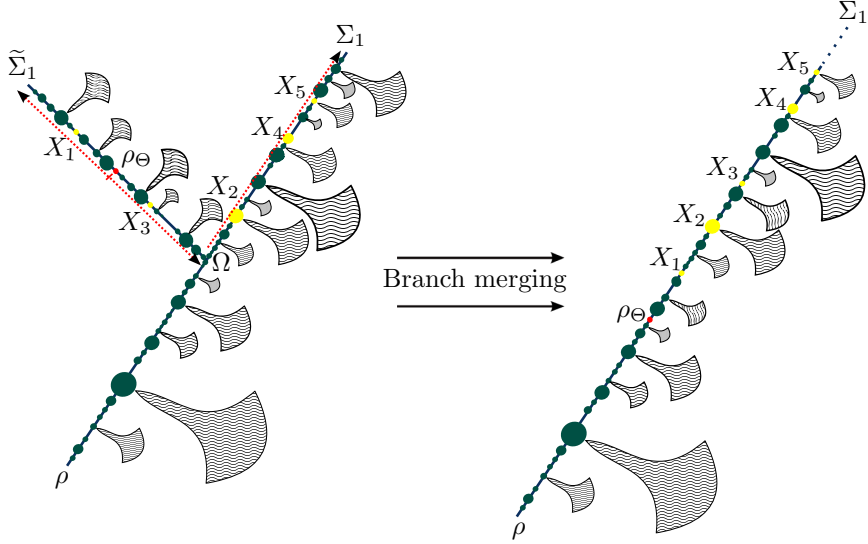  \caption{Branch merging on Ford CRTs. Reduced tree structure with subtrees attached in the atoms. Strings of beads are merged and subtrees are carried forward.} \label{pic4} \end{figure} 

\subsubsection{Application: the leaf embedding problem} 
 \label{recbm} 
In order to apply the branch merging procedure to the leaf embedding problem we first describe how to apply the branch merging procedure recursively on Ford CRTs.

\vspace{0.2cm} 
\textit{Recursive branch merging}.
Let $\mathcal T^{\alpha, 1- \alpha}$ be a Ford CRT for some $\alpha \in (0,1)$ as in Section \ref{BMFORD}. Consider $(\widetilde{\mathcal T}_k, \tilde{d}_k, \rho, \tilde{\mu}_k)$, $k \in \mathbb N$, with leaves $\Sigma_1,\ldots, \Sigma_k$ and root $\rho$, denoting the random $\mathbb R$-tree obtained by performing the branch merging operation given by Algorithm $\ref{bmalgo}$ in $(\mathcal T^{\alpha, 1- \alpha}, d, \rho, \mu)$, and in $k-1$ subsequently selected subtrees according to the following scheme, which explains how to obtain $(\widetilde{\mathcal T}_{k+1}, \tilde{d}_{k+1}, \rho, \tilde{\mu}_{k+1})$ from $(\widetilde{\mathcal T}_k, \tilde{d}_k, \rho, \tilde{\mu}_k)$, $k \in \mathbb N$.

Consider the subtree spanned by the leaves $\Sigma_1,\ldots, \Sigma_k$, denoted by $\widetilde{\mathcal R}_k:=\mathcal R(\widetilde{\mathcal T}_k, \Sigma_1,..,\Sigma_k)$, which is equipped with a random point measure $\mathcal P_{\widetilde{\mathcal R}_k}$ given by
\begin{equation*} \mathcal P_{\widetilde{\mathcal R}_k}:=\sum \limits_{x \in {\mathcal A ( \mu_{\widetilde{\mathcal R}_k})}} \delta_{(x, \mathcal S_x, \Pi_{x})} \end{equation*}
where the mass measure $\mu_{\widetilde{\mathcal R}_k}$ is the push-forward of $\tilde{\mu}_k$ via the projection
\begin{equation*} \pi^{{\widetilde{\mathcal R}_k}}: \widetilde{\mathcal T}_k \rightarrow {\widetilde{\mathcal R}_k}, \qquad \sigma \rightarrow \tilde{\Phi}^{(k)}_{\rho, \sigma}\left( \sup\{t \geq 0: \tilde{\Phi}^{(k)}_{\rho, \sigma}(t) \in \widetilde{\mathcal R}_k\} \right)\end{equation*} mapping the connected components of $\widetilde{\mathcal T}_k \setminus {\tilde{\mathcal R}_k}$ onto the closest point of the reduced tree ${\widetilde{\mathcal R}_k}$, the set $\mathcal A (\mu_{\widetilde{\mathcal R}_k})=\{x \in \widetilde{\mathcal R}_k: \mu_{\widetilde{\mathcal R}_k}(x) >0\}$ is the set of atoms of ${\mu}_{\widetilde{\mathcal R}_k}$, and $\tilde{\Phi}^{(k)}_{\rho, \sigma}: [0, \tilde{d}_k(\rho, \sigma)]\rightarrow \widetilde{\mathcal T}_k$, $\sigma \in \widetilde{\mathcal T}_k$, denote the unique geodesics associated with $\widetilde{\mathcal T}_k$. Suppose that the union of the sets $\Pi_{x}$, $x \in \mathcal A ( \mu_{\widetilde{\mathcal R}_k})$, is $\mathbb N_{\geq k+1}$, i.e. 
\begin{equation*}
\bigcup \limits_{x \in \mathcal A( \mu_{\widetilde{\mathcal R}_k})} \Pi_{x} = \mathbb N_{\geq k+1}.
\end{equation*}

Recursive application of the spinal decomposition theorem for binary fragmentation CRTs $\mathcal T^{\alpha,1-\alpha}$ (Lemma \ref{leafemblem}) yields that the point process $\mathcal P_{\widetilde{\mathcal R}_k}$ defines a labelled bead space in the following sense (see also \citep{1, 62} for subtree decompositions).

\begin{definition}[Labelled bead (space)]\rm \label{bead}  Let $(\mathcal T^{\alpha, 1-\alpha}, d, \rho, \mu)$ be a Ford CRT for some $\alpha \in (0,1)$. 
\begin{enumerate}
\item[(i)] A \textit{labelled bead} is a pair $(\mathcal S, \Pi_{\mathcal S})$ where $(\mathcal S, d_{\mathcal S})$ is a pointed metric space with a distinguished vertex $\rho_{\mathcal S}$ (the ``root'' of $\mathcal S$) and equipped with a mass measure $\mu_{\mathcal S}$, $\Pi_{\mathcal S}$ is an infinite label set, and, for $\Delta_{\mathcal S}:= \mu_{\mathcal S}(\mathcal S)$, 
\begin{equation*} (\mathcal S, \Delta_{\mathcal S}^{-\alpha}d_{\mathcal S}, \rho_{\mathcal S}, \Delta_{\mathcal S}^{-1}\mu_{\mathcal S}) \,{\buildrel d \over =}\, (\mathcal T^{\alpha, 1-\alpha}, d, \rho, \mu).\end{equation*}
\item[(ii)] A \textit{labelled bead space} is a pair $(\mathcal R, \mathcal P_{\mathcal R})$ where $(\mathcal R,d_{\mathcal R})$ is a metric space with a distinguished vertex $\rho_{\mathcal R}$ (the ``root'' of $\mathcal R$) and equipped with a discrete mass measure $\mu_{\mathcal R}$, and $\mathcal P_{\mathcal R}$ is a point process of the form
\begin{equation*}
\mathcal P_{\mathcal R}:= \sum \limits_{x \in  \mathcal  A(\mu_{\mathcal R})} \delta_{(x, \mathcal S_x, \Pi_{x})}
\end{equation*} where $\mathcal A(\mu_{\mathcal R})$ is the set of atoms of the measure $\mu_{\mathcal R}$ on $\mathcal R$, i.e. $\mathcal A(\mu_{\mathcal R}):=\{x \in \mathcal R: \mu_{\mathcal R}(x)>0\}$, such that $(\mathcal S_x, \Pi_{x})$ is a labelled bead for every $x \in \mathcal A( \mu_{\mathcal R})$ with root $\rho_{\mathcal S_x}:=x$, distance $d_{\mathcal S_x}$ and mass measure $\mu_{\mathcal S_x}$.
\end{enumerate}
\end{definition}

In order to obtain the branch linking $\widetilde{\mathcal R}_k$ and $\Sigma_{k+1}$ choose the unique labelled bead $({\mathcal S}_{x}, \Pi_{x})$ whose label set contains $k+1$, say $(\mathcal S_k, \Pi_k)$.
Note that $d \restriction_{\mathcal S_k}=\tilde{d}_{k} \restriction_{\mathcal S_k}$ and $\mu \restriction_{\mathcal S_k}=\tilde{\mu}_{k} \restriction_{\mathcal S_k}$, since each subtree related to a labelled bead on $\widetilde{\mathcal R}_k$ was present as a subtree of $(\widetilde{\mathcal T}_k, \tilde{d}_k, \rho)$ as well as in the initial CRT $(\mathcal T^{\alpha, 1-\alpha}, d, \rho)$.

We now consider the rescaled tree 
\begin{equation}  (\mathcal S_k, \tilde{d}_{\mathcal S_k},  \rho_k, \tilde{\mu}_{\mathcal S_k}):=(\mathcal S_k, \mu (\mathcal S_k)^{-\alpha} d\restriction_{\mathcal S_k}, \rho_k, \mu (\mathcal S_k)^{-1} \mu\restriction_{\mathcal S_k})\end{equation} completed by the root $\rho_k:=\pi^{{\widetilde{\mathcal R}_k}}(\mathcal S_k)$ and equipped with the label set $\Pi_k$. Since  \begin{equation*} (\mathcal S_k, \tilde{d}_{\mathcal S_k}, \rho_k, \tilde{\mu}_{\mathcal S_k}) \,{\buildrel d \over =}\,(T^{\alpha, 1-\alpha},d,\rho,\mu) \end{equation*} we can apply the procedure of Algorithm \ref{bmalgo} when interpreted appropriately.

\begin{enumerate}
\item[(i)] Choose the start configuration as in Algorithm \ref{bmalgo}(i) but label the three selected leaves by $(\Sigma_{k+1},\widetilde{\Sigma}_{k+1}, \Theta_{k})$ instead of $(\Sigma_{1},\widetilde{\Sigma}_{1}, \Theta_{k})$, and the branch point by $\Omega_k$ instead of $\Omega$. 

Furthermore, consider the label set $\Pi_k:=\{n_i, i \geq 1\}$ where subscripts are assigned according to the increasing bijection, i.e. $n_1=k+1$ and $n_2=k'$, where \begin{equation*} k':=\min \{ n \in \mathbb N: n \in \Pi_k \setminus \{k+1\}\}. \end{equation*} 

Determine a partition of the label sets $\Pi_k$ by considering a sequence of random variables $(Y_{n_i})_{i \geq 2}$ taking values in  $\mathcal R(\mathcal S_{k}, \Sigma_{k+1}, \widetilde{\Sigma}_{k+1})$ as in Algorithm \ref{bmalgo}(i) by proceeding in increasing order of labels $n_i, i \geq 2$. In particular, set $Y_{n_2}=Y_{k'}:=\rho_{\Theta_k}$ where $\rho_{\Theta_k}$ denotes the projection of $\Theta_k$ onto the reduced tree $\mathcal R(\mathcal S_{k}, \Sigma_{k+1}, \widetilde{\Sigma}_{k+1})$, and assign the labels $n_i, i \geq 3,$ to atoms $Y_{n_i}$ of $\mathcal R(\mathcal S_{k}, \Sigma_{k+1}, \widetilde{\Sigma}_{k+1})$ as in Algorithm \ref{bmalgo}(i) .

Algorithm \ref{bmalgo} applied to $(\mathcal S_k, \tilde{d}_{\mathcal S_k},\rho_k,  \tilde{\mu}_{\mathcal S_k})$ yields a CRT incorporating a spine transformation on the spine from $\rho_k$ to the leaf $\Sigma_{k+1}$. 
We obtain a labelled bead space ($]]\rho_k, \Sigma_{k+1}]],\mathcal P_{]]\rho_k,\Sigma_{k+1}]]})$,  i.e.
\begin{equation*}
\mathcal P_{]]\rho_k,\Sigma_{k+1}]]}:=\sum \limits_{x \in \mathcal A(\tilde{\mu}_{]]\rho_k, \Sigma_{k+1}]]})} \delta_{(x, \mathcal S_x, \Pi_{x})},
\end{equation*}
where $\tilde{\mu}_{]]\rho_k, \Sigma_{k+1}]]}$ denotes the mass measure on the merged spine $]]\rho_k, \Sigma_{k+1}]]$ whose atoms are the connected components of $\mathcal S_k \setminus ]]\rho_k, \Sigma_{k+1}]]$ projected onto $]]\rho_k, \Sigma_{k+1}]]$, and 
\begin{equation*} \mathcal A(\tilde{\mu}_{]]\rho_k, \Sigma_{k+1}]]})=\{x \in ]]\rho_k,\Sigma_{k+1}]]:\tilde{\mu}_{]]\rho_k, \Sigma_{k+1}]]}(x)>0 \} \end{equation*} is the associated set of atoms. Note that the union of the label sets $\Pi_{x}$, $x \in ]]\rho_k, \Sigma_{k+1}]]$, of the labelled beads aligned along $]]\rho_k, \Sigma_{k+1}]]$ is $\Pi_k \setminus \{k+1\}$.
\item[(ii)] Denote the transformed, scaled back tree by $(\mathcal S_k', \tilde{d}'_{\mathcal S_k'}, \rho_k, \tilde{\mu}'_{\mathcal S_k'} )$. The reduced tree spanned by the root $\rho$ and leaves $\Sigma_1,\ldots, \Sigma_{k+1}$ is obtained via 
\begin{equation*} \widetilde{\mathcal R}_{k+1}:=\widetilde{\mathcal R}_{k} \cup ]]\rho_k, \Sigma_{k+1}]]. \end{equation*} We replace the point $(\rho_k, \mathcal S_k, \Pi_k)$ by a new series of labelled beads aligned along $]]\rho_k, \Sigma_{k+1}]]$ to update the point process $\mathcal P_{\widetilde{\mathcal R}_{k}}$, and get
\begin{equation*} \mathcal P_{\widetilde{\mathcal R}_{k+1}}:=\sum \limits_{x \in \mathcal A(\mu_{\widetilde{\mathcal R}_{k+1}})} \delta_{(x, \mathcal S_x, \Pi_{x})},
\end{equation*}
where $\mathcal A(\mu_{\tilde{\mathcal R}_{k+1}})=\{x \in \widetilde{\mathcal R}_{k+1}:\mu_{\widetilde{\mathcal R}_{k+1}}(x) >0 \}$, and $\mu_{\widetilde{\mathcal R}_{k+1}}$ is uniquely determined via 
\begin{equation*} \mu_{\widetilde{\mathcal R}_{k+1}} \restriction_{]]\rho_k,\Sigma_{k+1}]]}=\tilde{\mu}_{]]\rho_k,\Sigma_{k+1}]]}, \quad \mu_{\widetilde{\mathcal R}_{k+1}} \restriction_{\widetilde{\mathcal R}_k \setminus \{\rho_k\}}=\mu_{\widetilde{\mathcal R}_{k}}, \quad \mu_{\widetilde{\mathcal R}_{k+1}}(\rho_k)=0. \end{equation*}
More precisely, we have $\mathcal P_{\widetilde{\mathcal R}_{k+1}} := \mathcal P_{\widetilde{\mathcal R}_k}-\delta_{(\rho_{\mathcal S_k}, \mathcal S_k, \Pi_k)}+\mathcal P_{]]\rho_k, \Sigma_{k+1}]]}$, i.e.
\begin{equation*}
\mathcal P_{\widetilde{\mathcal R}_{k+1}} 
=\sum \limits_{x \in \mathcal A(\mu_{\widetilde{\mathcal R}_{k}}), x \neq \rho_{\mathcal S_k}} \delta_{(x, \mathcal S_x, \Pi_x)}+\sum \limits_{x \in \mathcal A(\tilde{\mu}_{]]\rho_k, \Sigma_{k+1}]]})}\delta_{(x, \mathcal S_x, \Pi_x)}. \end{equation*}
Informally speaking, we replace $(\mathcal S_k, d \restriction_{\mathcal S_k}, \rho_k, {\mu}\restriction_{\mathcal S_k})$ by attaching the tree $(\mathcal S_k', \tilde{d}'_{\mathcal S_k'}, \rho_k, \tilde{\mu}'_{\mathcal S_k'})$ onto $\rho_k$, the branch point base point of the subtree $(\mathcal S_k, d \restriction_{\mathcal S_k}, \mu \restriction_{\mathcal S_k}, \Pi_k)$ on $\widetilde{\mathcal T}_k$, where  $(\mathcal S_k', \tilde{d}'_{\mathcal S_k'}, \rho_k, \tilde{\mu}'_{\mathcal S_k'})$ is obtained from the initial subtree $(\mathcal S_k, d \restriction_{\mathcal S_k}, \mu \restriction_{\mathcal S_k}, \Pi_k)$ by incorporating the distance changes related to the spine transformation and removing the two points $\Omega_k$ and $\widetilde{\Sigma}_{k+1}$ lost in the branch merging operation.

\item[(iii)] We define $(\widetilde{\mathcal T}_{k+1}, \tilde{d}_{k+1},\rho, \tilde{\mu}_{k+1})$ by
\begin{equation*}
\widetilde{\mathcal T}_{k+1}:= \widetilde{\mathcal R}_{k+1} \cup \bigcup \limits_{x \in \widetilde{\mathcal R}_{k+1}} \mathcal S_x
\end{equation*} 
where the distance $\tilde{d}_{k+1}$ is the push-forward of $\tilde{d}_{k}$ via the analogue of \eqref{distance}, and the mass measure $\tilde{\mu}_{k+1}$ is induced by $\tilde{\mu}_{k}$ as in \eqref{muone}, see Algorithm \ref{bmalgo} for the precise definition of $\tilde{d}_{k+1}$ and $\tilde{\mu}_{{k+1}}$ in the case $k=0$, which can be easily adapted to the general case $k \geq 0$.
\end{enumerate}

\begin{corollary}\label{corrbcrt} Let $(\mathcal T^{\alpha,1-\alpha}, d, \rho, \mu)$ be a Ford CRT for some $\alpha \in (0,1)$ as in Algorithm \ref{bmalgo}. The sequence of trees $((\widetilde{\mathcal T}_{k}, \tilde{d}_{{k}},\rho,\tilde{\mu}_{{k}}), k \geq 1)$ constructed recursively according to Algorithm \ref{bmalgo} and the recursive leaf embedding procedure, and the sequence of leaves $(\Sigma_k, k \geq 1)$ are such that 
\begin{equation*} (\mathcal R(\widetilde{\mathcal T}_k, \Sigma_1, \ldots, \Sigma_k), k \geq 1) \,{\buildrel d \over =}\, (\mathcal R^{\alpha, 2-\alpha}_k, k \geq 1), \end{equation*} where $\mathcal R^{\alpha,2-\alpha}_k$ is the almost-sure limit as $n \rightarrow \infty$ of the trees $\mathcal R(T_n^{\alpha,2-\alpha}, [k])$ scaled by $n^{-\alpha}$ associated with the $(\alpha, 2-\alpha)$-tree growth process as in Proposition \ref{propred}. 
\end{corollary}

\begin{proof} We prove Corollary \ref{corrbcrt} by induction. By Theorem \ref{basictransformation}, $\mathcal R(\widetilde{\mathcal T}_1, \Sigma_1)$ defines an $(\alpha, 2-\alpha)$-string of beads and is therefore distributed as $\mathcal R_1^{\alpha, 2-\alpha}$. Now, let $k \in \mathbb N$, and assume that  $\mathcal R(\widetilde{\mathcal T}_{k},\Sigma_1, \ldots, \Sigma_k)$ is distributed as $\mathcal R_k^{\alpha, 2-\alpha}$. Using the labelled bead space property, we conclude that the rescaled, transformed subtree \begin{equation*}(\mathcal S_k', \mu(\mathcal S_k')^{-\alpha}\tilde{d}'_{\mathcal S_k'}, \rho_k, \mu(\mathcal S_k)^{-1}\tilde{\mu}'_{\mathcal S_k'}) \end{equation*} has the same distribution as $(\widetilde{\mathcal T}_1, \tilde{d}_1, \rho, \tilde{\mu}_1)$, and that $\Sigma_{k+1}$ is such that \begin{equation*} \mathcal R(\mathcal \mu(S_k)^{-\alpha}S_k', \Sigma_{k+1})\,{\buildrel d \over =}\, \mathcal R(\widetilde{\mathcal T}_1, \Sigma_1). \end{equation*} Therefore, the subtree $(\mathcal S_k, \tilde{d}_k\restriction_{\mathcal S_k}, \rho_k, \tilde{\mu}_k\restriction_{\mathcal S_k})$ in $\widetilde{\mathcal T}_{k}$ is replaced by 
$(\mathcal S_k', \tilde{d}'_{\mathcal S_k'}, \rho_k, \tilde{\mu}'_{\mathcal S_k'})$  in order to obtain $(\widetilde{\mathcal T}_{k+1}, \tilde{d}_{{k+1}},\rho,\tilde{\mu}_{{k+1}})$, and 
\begin{equation*} \left(\mu(\mathcal S_k)^{-\alpha}\mathcal S_k', \tilde{d}'_{\mathcal S_k'}, \rho_k, \mu(\mathcal S_k)^{-1}\tilde{\mu}'_{\mathcal S_k'}\right)\,{\buildrel d \over =}\ \left(\widetilde{\mathcal T}_1, \tilde{d}_1, \rho, \tilde{\mu}_1\right). \end{equation*}
Since $(\widetilde{\mathcal T}_k, \tilde{d}_k, \rho, \tilde{\mu}_k)$ is such that $\mathcal R(\widetilde{\mathcal T}_k, \Sigma_1, \ldots, \Sigma_k)$ 
has the same distribution as $\mathcal R^{\alpha, 2-\alpha}_k$ by the induction hypothesis, and since, by construction, the projection of $\mathcal R(\widetilde{\mathcal T}_{k+1}, \Sigma_1, \ldots, \Sigma_{k+1})$ onto $\mathcal R(\widetilde{\mathcal T}_{k+1}, \Sigma_1, \ldots, \Sigma_{k})$ is $\mathcal R(\widetilde{\mathcal T}_{k}, \Sigma_1, \ldots, \Sigma_{k})$, the claim follows.
\end{proof}

\begin{remark} Note that by the refining nature of the branch merging algorithm,
\begin{equation*}
d_{\rm GH}\left(\widetilde{\mathcal T}_{k+1}, \mathcal R\left(\widetilde{\mathcal T}_{k+1}, \Sigma_1, \ldots, \Sigma_{k+1}\right)\right) \leq d_{\rm GH}\left(\widetilde{\mathcal T}_k, \mathcal R\left(\widetilde{\mathcal T}_k, \Sigma_1, \ldots, \Sigma_k\right)\right)\label{decreasing}
\end{equation*} 
for all $k \geq 1$. By Corollary \ref{corrbcrt}, the increasing limit
\begin{equation*} \widetilde{\mathcal T}_{\infty}=\overline{\bigcup_{k \geq 1} \mathcal R \left( \widetilde{\mathcal T}_k, \Sigma_1, \ldots, \Sigma_k \right)} \end{equation*} is a random compact $\mathbb R$-tree distributed as a Ford CRT (since this holds for the increasing sequence $(\mathcal R_k^{\alpha,2-\alpha}, k \geq 1))$. By construction,
\begin{equation} 
 \mathcal R\left(\widetilde{\mathcal T}_\infty, \Sigma_1, \ldots, \Sigma_k\right)=\mathcal R\left(\widetilde{\mathcal T}_k, \Sigma_1, \ldots, \Sigma_k\right). \label{reddi}
\end{equation}
Hence, $ \widetilde{\mathcal T}_{\infty}$ is a Ford CRT with all leaves $\Sigma_k, k \geq 1$, embedded. By the nature of the algorithm, \eqref{reddi} can be enriched by the corresponding point processes, i.e. 
\begin{equation*} 
\left( \mathcal R\left(\widetilde{\mathcal T}_\infty, \Sigma_1, \ldots, \Sigma_k\right),\mathcal P_{\mathcal R\left(\widetilde{\mathcal T}_\infty, \Sigma_1, \ldots, \Sigma_k\right)}\right)\,{\buildrel d \over =}\,\left( \mathcal R\left(\widetilde{\mathcal T}_k, \Sigma_1, \ldots, \Sigma_k\right), \mathcal P_{\mathcal R\left(\widetilde{\mathcal T}_k, \Sigma_1, \ldots, \Sigma_k\right)}\right). 
\end{equation*}
Based on a generalisation of Miermont's notion of $k$-marked trees \citep{38} to $\infty$-marked trees, it can be shown that $\widetilde{\mathcal T}_k \rightarrow \widetilde{\mathcal T}_\infty$ a.s. as $k \rightarrow \infty$ with all leaves $\Sigma_k, k \geq 1,$ embedded in $\widetilde{\mathcal T}_{\infty}$. See forthcoming work \citep{100} for different applications of $\infty$-marked trees.
\end{remark}

\section*{Acknowledgement} I would like to thank Matthias Winkel for introducing me to the topic of regenerative tree growth processes and for many fruitful discussions and comments. I would also like to thank the referee for providing very helpful suggestions to improve the quality of this paper.

\bibliographystyle{alea3}
\bibliography{BCRT}

\end{document}